\pdfoutput=1
\documentclass{article}

\usepackage{amsmath,amsfonts,amssymb,amsthm}
\usepackage{rotating}
\usepackage{comment}
\usepackage{tikz}
\usetikzlibrary{arrows,matrix}
\usepackage{url}
\usepackage[pdfauthor={Federico Galetto},pdftitle={Free resolutions of orbit closures for the representations associated to gradings on Lie algebras of type E6, F4 and G2}]{hyperref}

\newcommand{\rootnode}[3][above]{
	\begin{scope}[shift={#2}]
		\draw[very thick,fill=white] (0,0) circle (0.15);
		\node at (0,0) [label=#1:{\small #3}] {};
	\end{scope}
}

\newcommand{\singleline}[1]{
	\begin{scope}[shift={#1}]
		\draw[thick] (0,0) -- (2,0);
	\end{scope}
}

\newcommand{\singleverticalline}[1]{
	\begin{scope}[shift={#1}]
		\draw[thick] (0,0) -- (0,-2);
	\end{scope}
}

\newcommand{\doublelinearrowright}[1]{
	\begin{scope}[shift={#1}]
		\draw[thick] (0,0.05) -- (2,0.05);
		\draw[thick] (0,-0.05) -- (2,-0.05);
		\draw[very thick] (1.1,0) -- (0.9,0.15);
		\draw[very thick] (1.1,0) -- (0.9,-0.15);
	\end{scope}
}

\newcommand{\triplelinearrowright}[1]{
	\begin{scope}[shift={#1}]
		\draw[thick] (0,0.1) -- (2,0.1);
		\draw[thick] (0,0) -- (2,0);
		\draw[thick] (0,-0.1) -- (2,-0.1);
		\draw[very thick] (1.1,0) -- (0.9,0.15);
		\draw[very thick] (1.1,0) -- (0.9,-0.15);
	\end{scope}
}

\theoremstyle{definition}
\newtheorem{thm}{Theorem}[section]
\newtheorem{coro}[thm]{Corollary}
\newtheorem{prop}[thm]{Proposition}
\newtheorem{lemma}[thm]{Lemma}

\newtheorem{remark}[thm]{Remark}

\DeclareMathOperator{\Spin}{Spin}
\DeclareMathOperator{\GL}{GL}
\DeclareMathOperator{\SL}{SL}
\DeclareMathOperator{\SO}{SO}
\DeclareMathOperator{\Sp}{Sp}
\DeclareMathOperator{\Sym}{Sym}
\DeclareMathOperator{\Gr}{Gr}
\DeclareMathOperator{\Sc}{\mathbb{S}} 
\DeclareMathOperator{\coker}{coker}
\DeclareMathOperator{\rank}{rank}
\DeclareMathOperator{\depth}{depth}
\DeclareMathOperator{\height}{height}
\DeclareMathOperator{\codim}{codim}
\DeclareMathOperator{\disc}{disc}
\DeclareMathOperator{\im}{im}
\DeclareMathOperator{\id}{id}
\DeclareMathOperator{\sgn}{sgn}

\usepackage{mathabx,epsfig}
\def\acts{\mathrel{\reflectbox{$\righttoleftarrow$}}}

\renewcommand{\geq}{\geqslant}
\renewcommand{\leq}{\leqslant}

\newcommand{\QQ}{\mathbb{Q}}
\newcommand{\CC}{\mathbb{C}}
\newcommand{\ZZ}{\mathbb{Z}}
\newcommand{\tr}{tr}
\renewcommand{\gg}{\mathfrak{g}}
\newcommand{\Obar}{\overline{\mathcal{O}}}
\newcommand{\OO}{\mathcal{O}}
\newcommand{\CM}{Cohen-Macaulay}

\author{Federico Galetto}
\title{Free resolutions of orbit closures for the representations associated to gradings on Lie algebras of type $E_6$, $F_4$ and $G_2$}
\date{\today}

\begin{document}
\setcounter{MaxMatrixCols}{14}

\maketitle
\begin{abstract}
The irreducible representations of complex semisimple algebraic groups with finitely many orbits are parametrized by graded simple Lie algebras. For the exceptional Lie algebras, Kra\'skiewicz and Weyman exhibit the Hilbert polynomials and the expected minimal free resolutions of the normalization of the orbit closures. We present an interactive method to construct explicitly these and related resolutions in Macaulay2. The method is then used in the cases of the Lie algebras of type $E_6$, $F_4$, and $G_2$ to confirm the shape of the expected resolutions as well as some geometric properties of the orbit closures.
\end{abstract}

\tableofcontents

\section{Introduction}\label{intro}
In \cite{MR575790}, Kac classified the irreducible representations of semisimple groups whose nullcone contains finitely many orbits, the so-called visible representations. Among these representations are those which themselves contain finitely many orbits (when the representation coincides with its nullcone); these are referred to as representations of type I.

The representations of type I are parametrized by pairs $(X_n, x)$, where $X_n$ is a Dynkin diagram and $x$ is a ``distinguished'' node of $X_n$. The nodes of the Dynkin diagrams are numbered according to the conventions of Bourbaki \cite{MR0240238}. The choice of a distinguished node induces a grading on the root system $X_n$ and the corresponding simple Lie algebra $\gg$
	\[\gg =\bigoplus_{i\in\ZZ} \gg_i\]
satisfying the following:
\begin{itemize}
\item the Cartan subalgebra $\mathfrak{h}$ is contained in $\gg_0$;
\item the root space $\gg_\beta$ is contained in $\gg_i$, where $i$ is the coefficient of the simple root $\alpha$, corresponding the node $x$, in the expression of $\beta$ as a linear combination of simple roots;
\item the grading is compatible with the Lie bracket of $\gg$.
\end{itemize}
The representation for the pair $(X_n, x)$ is given by the vector space $\gg_1$ with the action of the group $G_0 \times \CC^\times$, where $G_0$ is the adjoint group of the Lie subalgebra $\gg_0$ of $\gg$; the second factor is the copy of $\CC^\times$ that occurs in the maximal torus of the adjoint group of $\gg$ but not in the maximal torus of $G_0$.

The orbit closures of the representations of type I were classified by Vinberg. In \cite{MR0430168}, he showed that the orbits are the irreducible components of the intersections of the nilpotent orbits in $\gg$ with the graded component $\gg_1$. Later, in \cite{MR549013}, Vinberg gave a combinatorial description of the orbits in terms of some graded subalgebras of $\gg$.

The study of the minimal free resolutions of these orbit closures goes back to Lascoux and his paper on determinantal varieties \cite{MR520233}. J\'ozefiak,  Pragacz and Weyman developed the case of determinantal ideals of symmetric and antisymmetric matrices \cite{MR646819}. The more general case of rank varieties for the classical types was studied by Lovett \cite{MR2309889}. For the exceptional types, Kra\'skiewicz and Weyman \cite{Kraskiewicz:2012yq} calculate the Hilbert functions of the normalization of the orbit closures. They also describe the terms of the expected resolutions of the normalizations as representations of the group acting.

The goal of this work is to introduce a framework for constructing the complexes of Kra\'skiewicz and Weyman explicitly. This involves computational methods, carried out in the software system Macaulay2 \cite{Grayson:uq}, and techniques from representation theory, commutative algebra and algebraic geometry. As a result, we can verify the shape of the free resolutions conjectured by Kra\'skiewicz and Weyman in many cases, whenever our computations are feasible; moreover, we obtain the minimal free resolutions of the coordinate rings of the orbit closures and confirm results about Cohen-Macaulay and Gorenstein orbit closures. Finally, we establish containment and singularity of the orbit closures. In particular, we examine the exceptional Lie algebras of type $E_6$, $F_4$ and $G_2$, and provide a detailed analysis of each orbit closure for those cases.

In the next section, we provide an overview of the geometric technique used in the work of Kra\'skiewicz and Weyman to extract information about the orbit closures. In section \ref{sec:computations}, we describe the computations ran in Macaulay2; in particular, we introduce the \emph{interactive method for syzygies} and the \emph{cone procedure} for non normal orbit closures. In section \ref{sec:exactness}, we discuss the \emph{equivariant exactness criterion}, a simple method to establish exactness of our complexes. In section \ref{sec:equations}, we explain how to obtain the defining equations for the orbit closures and verify they generate a radical ideal.
Finally, the last sections are devoted to a detailed analysis of the orbit closures for the representations associated with gradings on the Lie algebras of type $E_6$, $F_4$ and $G_2$. The appendix contains a brief description of the equivariant maps used to construct the differentials in our complexes.

This work is accompanied by a collection of Macaulay2 files containing the defining ideals for the orbit closures as well as presentations of related modules. All files are available online at \href{http://www.mast.queensu.ca/~galetto/orbits}{\url{http://www.mast.queensu.ca/~galetto/orbits}}.

\section{Geometric technique}
The work of Kra\'skiewicz and Weyman relies on the use of the geometric technique of calculating syzygies. We provide here a brief overview of the technique applied to the context of our work. For more details, the reader can consult \cite{MR1988690}.

Let us fix a representation with finitely many orbits; using the notation of the introduction, we have an action $G_0 \acts \gg_1$. Consider $\gg_1$ as an affine space $\mathbb{A}^N_\CC$; every orbit closure $\Obar$ in $\gg_1$ is an affine variety in $\mathbb{A}^N_\CC$. Moreover, every orbit closure $\Obar$ admits a desingularization $Z$ which is the total space of a homogeneous vector bundle $\mathcal{S}$ on some homogeneous space $G_0/P$, for some parabolic subgroup $P$ of $G_0$. The space $\mathbb{A}^N_\CC \times G_0/P$ can be viewed as the total space of a trivial vector bundle $\mathcal{E}$ of rank $N$ over $G_0/P$ and $\mathcal{S}$ is a subbundle of $\mathcal{E}$. Altogether we have the following picture:
	\begin{center}
	\begin{tikzpicture}
			\matrix (m) [matrix of math nodes, row sep=20pt, column sep=30pt, text height=1.5ex, text depth=0.25ex, ampersand replacement=\&]
			{ Z \& \mathbb{A}^N \times G_0 / P \\
			\overline{\mathcal{O}} \& \mathbb{A}^N \\ };
			\path[->,font=\scriptsize]
			(m-1-1) edge node[auto] {$q'$} (m-2-1)
			(m-1-2) edge node[auto] {$q$} (m-2-2);
			\path[right hook->,font=\scriptsize]
			(m-1-1) edge node[auto] {} (m-1-2)
			(m-2-1) edge node[auto] {} (m-2-2);
	\end{tikzpicture}
	\end{center}
where $q$ is the projection on $\mathbb{A}^N$ and $q'$ is its restriction to $Z$.

Let $A=\CC [\mathbb{A}^N_\CC]$; this is the polynomial ring over which we will carry out all computations. Also, introduce the vector bundle $\xi = (\mathcal{E}/\mathcal{S})^*$ on $G_0/P$. We can now state the basic theorem \cite[Thm. 5.1.2]{MR1988690}. 

\begin{thm}
Define the graded free $A$-modules:
\[ F_i : = \bigoplus_{j \geq 0} H^j \left( G_0 / P, \textstyle \bigwedge^{i+j} \xi \right) \otimes_\CC A (- i - j).\]
There exist minimal $G_0$-equivariant differentials
	\[d_i \colon F_i \longrightarrow F_{i-1}\]
of degree 0 such that $F_\bullet$ is a complex of graded free $A$-modules with
	\[H_{-i} (F_\bullet) = \mathcal{R}^i q_* \mathcal{O}_Z.\]
In particular, $F_\bullet$ is exact in positive degrees.
\end{thm}

Recall that $Z$ is a desingularization of $\Obar$; in particular, the map $q' \colon Z\rightarrow \Obar$ is a birational isomorphism. The next theorem \cite[Thm. 5.1.3]{MR1988690} gives a criterion for the complex $F_\bullet$ to be a free resolution of the coordinate ring of $\Obar$.

\begin{thm}
The following properties hold.
	\begin{enumerate}
	\item The module $q'_* \mathcal{O}_Z$ is the normalization of $\CC [\Obar]$.
	\item If $\mathcal{R}^i q_* \mathcal{O}_Z = 0$ for $i>0$, then $F_\bullet$ is a finite free resolution of the normalization of $\CC [\Obar]$ as an $A$-module.
	\item If $\mathcal{R}^i q_* \mathcal{O}_Z = 0$ for $i>0$ and $F_0 = A$, then $\Obar$ is normal and has rational singularities.
	\end{enumerate}
\end{thm}

The modules $F_i$ are defined in terms of cohomology groups of the bundles $\bigwedge^{i+j} \xi$ on the flag variety $G_0 / P$. These cohomology groups can be computed directly using Bott's algorithm when $\xi$ is semisimple \cite[Thm. 4.1.4]{MR1988690}. When $\xi$ is not semisimple, its structure becomes complicated so working with it directly is more difficult. However it is still possible to calculate the equivariant Euler characteristic of the bundles $\bigwedge^{i+j} \xi$. This was done by Kra\'skiewicz and Weyman in \cite{Kraskiewicz:2012yq}; their work provides the Hilbert functions of the normalization of the orbit closures together with an estimate for the shape of the minimal free resolution of their coordinate rings in terms of $G_0$-modules and equivariant maps.

\section{Computations}\label{sec:computations}
The expected resolution $F_\bullet$ of the coordinate ring of $\Obar$ is constructed using the information provided by the equivariant Euler characteristic. As such it may be missing those syzygies the Euler characteristic was unable to detect due to cancellation; namely all those corresponding to an irreducible representation occurring in neighboring homological degrees and in the same homogeneous degree of the resolution. By constructing $F_\bullet$ explicitly in Macaulay2 (M2) \cite{Grayson:uq}, we can ensure it is the actual resolution for $\Obar$.

In this section we describe the type of computations that were carried out. All computations were run in M2. Although M2 can compute minimal free resolutions directly, this is generally impractical for many of the examples we consider, given the amount of computational resources and time needed for the algorithms to produce any result. Still, for some of the smaller examples, we were able to obtain resolutions directly or using the options {\tt DegreeLimit} and {\tt LengthLimit} to aid the computation.

\subsection{The interactive method for syzygies}\label{sec:interactive}
Let $F_\bullet$ be the expected resolution for $\Obar$ and suppose that some differential $d_i \colon F_i \rightarrow F_{i-1}$ can be written explicitly as a matrix with entries in the polynomial ring $A$. Using M2 and the command {\tt syz}, one can find the first syzygies of $d_i$. Since we expect these to coincide with the differential $d_{i+1}$, we use the expected degree of $d_{i+1}$ as a degree bound with the option {\tt DegreeLimit} to speed up the computation. This procedure can be iterated to recover the tail of the expected resolution:
	\begin{center}
	\begin{tikzpicture}
			\matrix (m) [matrix of math nodes, row sep=20pt, column sep=30pt, text height=1.5ex, text depth=0.25ex, ampersand replacement=\&]
			{\coker d_i \& F_{i-1} \& F_i \& T_\bullet\\};
			\path[<-,font=\scriptsize]
			(m-1-1) edge node[auto] {} (m-1-2)
			(m-1-2) edge node[auto] {$d_i$} (m-1-3)
			(m-1-3) edge node[auto] {} (m-1-4);
	\end{tikzpicture}
	\end{center}
Similarly, transposing $d_i$ and calculating syzygies gives the complex
	\begin{center}
	\begin{tikzpicture}
			\matrix (m) [matrix of math nodes, row sep=20pt, column sep=30pt, text height=1.5ex, text depth=0.25ex, ampersand replacement=\&]
			{H_\bullet \& F_{i-1}^* \& F_i^* \& \coker d_i^*\\};
			\path[->,font=\scriptsize]
			(m-1-1) edge node[auto] {} (m-1-2)
			(m-1-2) edge node[auto] {$d_i^*$} (m-1-3)
			(m-1-3) edge node[auto] {} (m-1-4);
	\end{tikzpicture}
	\end{center}
which can be dualized to obtain the head of the expected resolution. Splicing these two complexes together, we obtain the expected resolution in the form:
	\begin{center}
	\begin{tikzpicture}
			\matrix (m) [matrix of math nodes, row sep=20pt, column sep=30pt, text height=1.5ex, text depth=0.25ex, ampersand replacement=\&]
			{H_\bullet^* \& F_{i-1} \& F_i \& T_\bullet\\};
			\path[<-,font=\scriptsize]
			(m-1-1) edge node[auto] {} (m-1-2)
			(m-1-2) edge node[auto] {$d_i$} (m-1-3)
			(m-1-3) edge node[auto] {} (m-1-4);
	\end{tikzpicture}
	\end{center}

Clearly, taking syzygies with degree bounds and dualizing may cause the resulting complex not to be exact; we prove exactness with the methods described in section \ref{sec:exactness}. We refer to this method of constructing the expected resolution as the \emph{interactive method} for calculating syzygies.

\begin{remark}
In most cases, given the decomposition into irreducible representations of the modules $F_i$, the differentials $d_i$ are uniquely determined by Schur's lemma, up to a choice of scalars. However writing them explicitly is in general quite complicated. When the defining equations, i.e. the differential $d_1$, are not known, our choice of differential $d_i$ to write explicitly falls on those matrices that are easier to describe (e.g. matrices with linear entries).
\end{remark}

\subsection{The cone procedure}\label{cone_procedure}
When the orbit closure $\overline{\mathcal{O}}$ is not normal, the geometric technique returns the expected resolution of the coordinate ring of the normalization $\mathcal{N}(\overline{\mathcal{O}})$, as an $A$-module. We have an exact sequence of $A$-modules:
	\begin{center}
	\begin{tikzpicture}
			\matrix (m) [matrix of math nodes, row sep=20pt, column sep=25pt, text height=1.5ex, text depth=0.25ex, ampersand replacement=\&]
			{0 \& \CC [\overline{\mathcal{O}}] \& \CC [\mathcal{N}(\overline{\mathcal{O}})] \& C \& 0\\};
			\path[->,font=\scriptsize]
			(m-1-1) edge node[auto] {} (m-1-2)
			(m-1-3) edge node[auto] {} (m-1-4)
			(m-1-4) edge node[auto] {} (m-1-5);
			\path[right hook->,font=\scriptsize]
			(m-1-2) edge node[auto] {} (m-1-3);
	\end{tikzpicture}
	\end{center}

Using the interactive method, we recover the expected resolution $F_\bullet$ for $\mathcal{N}(\overline{\mathcal{O}})$; in particular, $d_1 \colon F_1\rightarrow F_0$ is a minimal presentation of $\CC [\mathcal{N}(\overline{\mathcal{O}})]$. As a $G_0$-representation, the module $C$ is always obtained from $\CC [\mathcal{N}(\overline{\mathcal{O}})]$ by removing some irreducible representations. This implies that, in the appropriate bases for $F_0$ and $F_1$, a presentation of $C$ is given by a map of free modules $d'_1 : F'_1 \rightarrow F'_0$ whose matrix is obtained from that of $d_1$ by dropping some rows and columns. This presentation can be used in M2 to construct a resolution $F'_\bullet$ for $C$ (using the interactive method, if necessary). Moreover, the projection $\pi_0 \colon F_0 \rightarrow F'_0$ lifts to a map of complexes $\tilde{\pi} \colon F_\bullet \rightarrow F'_\bullet$; this can be achieved explicitly in M2 with the command {\tt extend} or via a step by step factorization with the command {\tt //}. It is well known that the cone of $\tilde{\pi}$ is a (non necessarily minimal) free resolution of $\CC [\overline{\mathcal{O}}]$; in M2, this is recovered with {\tt cone}. Given the shift in homological degree introduced by the cone, $\CC [\overline{\mathcal{O}}]$ is the degree one homology of the cone of $\tilde{\pi}$. Now $\CC [\overline{\mathcal{O}}]$ can be minimized and resolved directly or using the interactive method. We call this the \emph{cone procedure}.

\begin{remark}
A more efficient version of this technique comes from the realization that, because the homology of the cone of $\tilde{\pi}$ is concentrated in degree one, it is not necessary to construct the whole cone in M2. In fact it is enough to know the resolutions of $\CC [\mathcal{N}(\overline{\mathcal{O}})]$ and $C$ up to homological degree two and use them to construct the part of the cone needed to recover the homology in degree one. We will refer to this as the \emph{truncated cone procedure}.
\end{remark}

\section{Exactness}\label{sec:exactness}
Once the expected resolution has been constructed in M2, we turn to the issue of proving it is exact.

\subsection{The equivariant exactness criterion}\label{exactness_criterion}
Recall the exactness criterion of Buchsbaum and Eisenbud \cite[Thm. 20.9]{MR1322960}:
	\begin{thm}
	Let $A$ be a ring and
	\begin{center}
	\begin{tikzpicture}
			\matrix (m) [matrix of math nodes, row sep=20pt, column sep=22pt, text height=1.5ex, text depth=0.25ex, ampersand replacement=\&]
			{F_\bullet : \& F_0 \& F_1 \& \ldots \& F_{n-1} \& F_n\\};
			\path[<-,font=\scriptsize]
			(m-1-2) edge node[auto] {$d_1$} (m-1-3)
			(m-1-3) edge node[auto] {} (m-1-4)
			(m-1-4) edge node[auto] {} (m-1-5)
			(m-1-5) edge node[auto] {$d_n$} (m-1-6);
	\end{tikzpicture}
	\end{center}
a complex of free $A$-modules.
	$F_\bullet$ is exact if and only if $\forall k = 1, \ldots, n$:
		\begin{enumerate}
		\item $\rank (F_k) = \rank (d_k) + \rank (d_{k+1})$;
		\item $\depth (I (d_k)) \geq k$ where $I (d_k)$ is the ideal of $A$ generated by maximal non vanishing minors of $d_k$.
		\end{enumerate}
	\end{thm}
The map $d_{n+1}$ is understood to be the zero map.

When $A$ is our polynomial ring and $F_\bullet$ is the expected resolution of an orbit closure $\Obar$, we have a more efficient way of proving exactness relying on the fact that $F_\bullet$ is $G_0$-equivariant. Recall that $G_0$ acts on $\mathbb{A}^N_\CC$, identified with $\gg_1$, with finitely many orbits $\mathcal{O}_0, \ldots, \mathcal{O}_t$; it  follows immediately that the action has a dense orbit, also referred to as the generic orbit. 
Let $p_j$ be a representative of the orbit $\mathcal{O}_j$ and let $p$ be a representative of the generic orbit.

\begin{lemma}\label{lemma:eec}
Let $d:F\rightarrow F'$ be a non zero, $G_0$-equivariant, minimal map of graded free $A$-modules of finite rank.
	\[\depth (I(d)) = \min \{ \codim (\Obar_j) \mid \rank (d|_{p_j}) < \rank (d|_{p}) \}.\]
\end{lemma}

Notice that the minimum on the right hand side is taken over a non empty set, since the rank of $d$ at the origin is zero.

\begin{proof}
Denote $\mathcal{V} (I(d))$ the zero set in $\mathbb{A}^N_\CC$ of the ideal $I(d)$ of $A$. We have
	\[\depth (I(d)) = \depth (\sqrt{I(d)}) = \codim (\mathcal{V} (I(d)))\]
where the first equality holds because radicals preserve depth and the second equality follows from the fact the that the polynomial ring $A$ is Cohen-Macaulay. By assumption, $d$ is $G_0$-equivariant, hence $I(d)$ is $G_0$-equivariant and so is its zero set. Therefore we can write
	\[\mathcal{V} (I(d)) = \bigcup_{\rank (d|_{p_j}) < \rank (d|_{p})} \mathcal{O}_j\]
which is a finite union of $G_0$-orbits. Finally
	\[\codim (\mathcal{V} (I(d))) = \min \{ \codim (\Obar_j) \mid \rank (d|_{p_j}) < \rank (d|_{p}) \}. \]
\end{proof}

\begin{prop}[Equivariant exactness criterion]\label{eec}
Let $F_\bullet$ be a $G_0$-equivariant minimal complex of graded free $A$-modules and assume the differentials $d_i : F_i \rightarrow F_{i-1}$ are non zero. Then $F_\bullet$ is exact if and only if $\forall k = 1, \ldots, n$:
	\begin{enumerate}
	\item\label{cond1} $\rank (F_k) = \rank (d_k|_{p}) + \rank (d_{k+1}|_{p})$;
	\item\label{cond2} $\min \{ \codim (\Obar_j) \mid \rank (d_k|_{p_j}) < \rank (d_k|_{p}) \} \geq k$.
	\end{enumerate}
\end{prop}

\begin{proof}
Condition \ref{cond1} is equivalent to the first condition in the criterion of Buchsbaum and Eisenbud. This is because $p$ is a representative of the generic orbit and therefore $\rank (d_k|_p) = \rank (d_k)$.

As for condition \ref{cond2}, this is equivalent to the second condition in the criterion of Buchsbaum and Eisenbud, because, by lemma \ref{lemma:eec}, we have
	\[\depth (I(d_k)) = \min \{ \codim (\Obar_j) \mid \rank (d_k|_{p_j}) < \rank (d_k|_{p}) \} \geq k.\]
\end{proof}

\subsection{Dual complexes and Cohen-Macaulay orbits}

A useful feature of the equivariant exactness criterion is that it easily allows to check if the dual complex $F_\bullet^*$ is exact, providing us with the following.

\begin{coro}
Under the hypotheses of proposition \ref{eec}, suppose $\exists j_0 \in \{0,\ldots,t\}$ such that whenever $\rank (d_k|_{p_j}) < \rank (d_k|_{p})$, for some $k$ and $j$, we have:
	\begin{itemize}
	\item $\rank (d_k|_{p_{j_0}}) < \rank (d_k|_{p})$;
	\item $\codim (\Obar_{j_0}) \leq \codim (\Obar_j)$.
	\end{itemize}
Then $F_\bullet$ is exact if and only if $F^*_\bullet$ is exact.
\end{coro}

\begin{remark}
The conditions in the corollary can be simply restated by saying that the ranks of the differentials $d_k$ drop simultaneously at the same orbit $\Obar_{j_0}$.
\end{remark}

\begin{proof}
First observe that $\rank (d^*_k|_{p_j}) = \rank (d_k|_{p_j})$. Now apply the equivariant exactness criterion. Condition \ref{cond1} is trivially satisfied. Condition \ref{cond2} together with the hypothesis implies
	\[\codim (\Obar_{j_0})=\min \{ \codim (\Obar_j) \mid \rank (d_k|_{p_j}) < \rank (d_k|_{p}) \} \geq k\]
for all $k$. Since the left hand side is independent of $k$, condition \ref{cond2} must be satisfied for $F_\bullet$ and $F^*_\bullet$ at the same time.
\end{proof}

The above corollary implies that when $F_\bullet$ is exact and the ranks of the differentials $d_k$ drop simultaneously at the same orbit,  $F_\bullet$ resolves a perfect module. In particular, if $F_\bullet$ resolves the coordinate ring of some orbit closure $\Obar$, then $\Obar$ is Cohen-Macaulay.

\begin{remark}
As detailed in the discussion of the interactive method \ref{sec:interactive}, the head $H_\bullet$ of the expected resolution is obtained by transposing the matrix of a differential $d_i$, resolving its cokernel and then dualizing back. It was noted then that taking duals could affect exactness. However, it follows immediately from our previous corollary, that if $\Obar$ is Cohen-Macaulay both $H_\bullet$ and $H_\bullet^*$ are exact.
\end{remark}

\section{Equations of the orbit closures}\label{sec:equations}
Once we establish that the expected resolution $F_\bullet$ for an orbit closure $\Obar$ is exact, the entries  of the first differential $d_1 \colon F_1 \rightarrow F_0$ generate an ideal $I$ of $A$ whose zero set $\mathcal{V} (I)$ is precisely $\Obar$. In this section we address the issue of determining if $I$ is a radical ideal. We also describe a simple criterion to establish inclusion and singularity of the orbit closures.

\subsection{The coordinate ring of an orbit closure}\label{reducedness}
The generators of the ideal $I$ provide equations for the variety $\Obar$. There is no guarantee however that $I$ is radical; equivalently, $A/I$ need not be reduced, hence it may not be the coordinate ring of $\Obar$.

We outline here the method used to show that the ring $R=A/I$ is indeed reduced. The proof relies on the following characterization of reduced rings \cite[Ex. 11.10]{MR1322960} which is an analogue of Serre's criterion for normality.

\begin{prop}\label{Serre}
A Noetherian ring $R$ is reduced if and only if it satisfies:
	\begin{list}{}{}
	\item[$(R_0)$:] the localization of $R$ at each prime of height 0 is regular;
	\item[$(S_1)$:] all primes associated to zero have height 0.
	\end{list}
\end{prop}

We address the condition $(S_1)$ first. In our case, $\Obar$ is always irreducible since it is an orbit closure for an irreducible group $G_0$. This implies that the ideal $I$ has a unique minimal prime, namely $\sqrt{I}$. Therefore the condition $(S_1)$ is equivalent to $\sqrt{I}$ being the only associated prime of $I$. In particular, $I$ will have no embedded primes.

To prove that $\sqrt{I}$ is the only associated prime of $I$, we adapt a result from \cite[Cor. 20.14]{MR1322960}.

\begin{lemma}\label{ass_primes_free_res}
Let
\begin{center}
	\begin{tikzpicture}
			\matrix (m) [matrix of math nodes, row sep=20pt, column sep=22pt, text height=1.5ex, text depth=0.25ex, ampersand replacement=\&]
			{F_\bullet : \& F_0 \& F_1 \& \ldots \& F_{n-1} \& F_n \& 0\\};
			\path[<-,font=\scriptsize]
			(m-1-2) edge node[auto] {$d_1$} (m-1-3)
			(m-1-3) edge node[auto] {} (m-1-4)
			(m-1-4) edge node[auto] {} (m-1-5)
			(m-1-5) edge node[auto] {$d_n$} (m-1-6)
			(m-1-6) edge node[auto] {} (m-1-7);
	\end{tikzpicture}
	\end{center}
be a free resolution of $R=A/I$ as an $A$-module and suppose $\depth_A (I)=c$. We have $\depth_{A_\mathfrak{p}} (\mathfrak{p} A_\mathfrak{p}) = c$ for all associated primes $\mathfrak{p}$ of $I$ if and only if $\depth (I(d_k)) > k$ for all $k>c$, where $I (d_k)$ is the ideal of $A$ generated by the maximal non vanishing minors of $d_k$.
\end{lemma}

The following result provides a simple test for the condition $(S_1)$ in all cases we consider.

\begin{prop}\label{S1_test}
Let $I$ be an ideal in the polynomial ring $A$ such that the zero locus $\mathcal{V} (I)$ of $I$ is irreducible of codimension $c$. Let
\begin{center}
	\begin{tikzpicture}
			\matrix (m) [matrix of math nodes, row sep=20pt, column sep=22pt, text height=1.5ex, text depth=0.25ex, ampersand replacement=\&]
			{F_\bullet : \& F_0 \& F_1 \& \ldots \& F_{n-1} \& F_n \& 0\\};
			\path[<-,font=\scriptsize]
			(m-1-2) edge node[auto] {$d_1$} (m-1-3)
			(m-1-3) edge node[auto] {} (m-1-4)
			(m-1-4) edge node[auto] {} (m-1-5)
			(m-1-5) edge node[auto] {$d_n$} (m-1-6)
			(m-1-6) edge node[auto] {} (m-1-7);
	\end{tikzpicture}
	\end{center}
be a free resolution of $R=A/I$ as an $A$-module. If $\depth (I(d_k)) > k$ for all $k>c$, then $\sqrt{I}$ is the only associated prime of $I$ and $R$ satisfies $(S_1)$.
\end{prop}

\begin{proof}
First observe that
\[c=\codim (\mathcal{V} (I)) = \depth_A (\sqrt{I}) = \depth_A (I),\]
because the ring $A$ is \CM{} and taking radicals preserve depth.

Now let $\mathfrak{p}$ be an associated prime of $I$. By lemma \ref{ass_primes_free_res}, we have $\depth_{A_\mathfrak{p}} (\mathfrak{p} A_\mathfrak{p}) = c$. Since $\sqrt{I}$ is the unique minimal prime of $I$, the inclusion $\sqrt{I} \subseteq \mathfrak{p}$ holds. We deduce that
\[c = \depth_A (\sqrt{I}) \leq \depth_A (\mathfrak{p}) \leq \depth_{A_\mathfrak{p}} (\mathfrak{p} A_{\mathfrak{p}}) = c,\]
where the last inequality holds because passing to the localization preserves regular sequences. Since $A$ is \CM{}, it follows that
\[\height (\mathfrak{p}) = \depth_A (\mathfrak{p}) = c.\]
But $\sqrt{I} \subseteq \mathfrak{p}$ and both ideals are primes of height $c$. Therefore $\mathfrak{p} = \sqrt{I}$.
\end{proof}

\begin{remark}
The test given in \ref{S1_test} is quite easy to apply in practice. The (co)dimension of each orbit closure is known a priori. Moreover, as observed in the proof of proposition \ref{eec}, 
\[\depth (I(d_k)) = \min \{ \codim (\Obar_j) \mid \rank (d_k|_{p_j}) < \rank (d_k|_{p}) \},\]
where $p_j$ is a representative of the orbit $\OO_j$ and $p$ is a representative of the dense orbit.
\end{remark}

We now turn to the condition $(R_0)$ in \ref{Serre}.

\begin{prop}\label{reducedness_criterion}
Let $I$ be an ideal in the polynomial ring $A$ with zero locus $\mathcal{V} (I)$ of  codimension $c$. Assume also that $\mathcal{V} (I)$ is irreducible and that $\sqrt{I}$ is the only associated prime of $I$. Denote $J$ the Jacobian matrix for a set of generators of $I$. If there exists a point $x\in\mathcal{V}(I)$ such that $\rank (J|_x) = c$, then the ring $R=A/I$ satisfies the condition $(R_0)$ and is therefore reduced.
\end{prop}

\begin{proof}
Let $\mathfrak{m}$ be the maximal ideal of $A$ corresponding to the point $x\in\mathcal{V} (I)$ and let $\mathfrak{p} = \sqrt{I}$. We denote $\overline{\mathfrak{m}}$ and  $\overline{\mathfrak{p}}$ respectively, the images of $\mathfrak{m}$ and  $\mathfrak{p}$ under the canonical projection $A \rightarrow A/I=R$. Clearly $R$ satisfies $(S_1)$, given the hypothesis on the associated primes of $I$. To prove $R$ satisfies $(R_0)$ we must show that the local ring $R_{\overline{\mathfrak{p}}}$ is regular, since $\overline{\mathfrak{p}}$ is the only height 0 prime in $R$.

First we observe that the ring $R_{\overline{\mathfrak{m}}}$ is regular by the Jacobian criterion \cite[Th. 16.19]{MR1322960}. By the transitivity of localization
	\[R_{\overline{\mathfrak{p}}} \cong (R_{\overline{\mathfrak{m}}})_{\overline{\mathfrak{p}} R_{\overline{\mathfrak{m}}}}.\]
The localization of a regular local ring at a prime ideal is regular. Hence $R_{\overline{\mathfrak{p}}}$ is regular.
\end{proof}

\begin{remark}
In practice, for an orbit closure $\Obar = \mathcal{V} (I)$, we will check that the rank of the Jacobian matrix of $I$ is equal to $\codim (\Obar)$ at a representative $x$ of $\OO$. Indeed the point $x$ is smooth in $\Obar$. This is because the singular locus of $\Obar$ is equivariant; therefore it must be an orbit closure of codimension at least one in $\Obar$ and cannot contain $x$.
\end{remark}

\begin{remark}
A \CM{} ring $R$ is generically reduced (i.e. its localization at each minimal prime is reduced) if and only if it is reduced \cite[Ex. 18.9]{MR1322960}. Notice how generically reduced is precisely the condition $(R_0)$.
In fact, if our ring $R=A/I$ is \CM{}, then it automatically satisfies $(S_1)$. Therefore it is enough to check $(R_0)$ as outlined above.
\end{remark}

\subsection{Inclusions and singular loci of orbit closures}

One problem that can be answered easily once we have equations for the orbit closures is  to determine how they sit one inside the other. To be more precise, we introduce the degeneration partial order by setting $\mathcal{O}_i \leq \mathcal{O}_j$ if and only if $\mathcal{O}_i \subseteq \Obar_j$. Because we have finitely many orbits, the entire picture can be described by checking if the equations of $\Obar_j$ vanish at the representative $p_i$ of $\mathcal{O}_i$. This can be achieved conveniently in M2.

Next we can use the equations of an orbit closure to construct a Jacobian matrix and apply the Jacobian criterion to determine the singular locus of the orbit closure \cite[p. 31]{MR0463157}. Once again, it is enough to evaluate the rank of the Jacobian matrix at finitely many points, the points being representatives of the orbits.


The information on containment and singularity of the orbit closures is presented in a table with rows labeled by the orbits and columns labeled by their closures. The cell corresponding to a row $\OO_i$ and a column $\Obar_j$ can be empty, meaning the points of the orbit $\OO_i$ are not contained in the orbit closure $\Obar_j$, or it can contain the letters `ns' (respectively `s') to indicate that the points of $\OO_i$ are non singular (respectively singular) in the orbit closure $\Obar_j$.

\subsection{The degenerate orbits}
Let $(X_n, x)$ be a Dynkin diagram with a distinguished node. The Lie algebra $\gg$ of type $X_n$ has a grading
\[ \gg = \bigoplus_{i\in \ZZ} \gg_i\]
induced by the choice of the distinguished node $x$, as described in the introduction. Let $\OO$ be an orbit for the action $G_0 \times \CC^\times \acts \gg_1$. Suppose there is a node $y\neq x$ in $X_n$ such that $\OO \cap \gg'_1 \neq 0$, where $\gg'$ is the graded subalgebra of $\gg$ corresponding to the subdiagram $(X_n \setminus \{x\}, y)$. Then we say the orbit $\OO$ is degenerate. This means that $\OO$ comes from an orbit $\OO'$ that occurs in the case of a smaller Dynkin diagram. The paper of Kra\'skiewicz and Weyman \cite{Kraskiewicz:2012yq} describes a method to obtain the free resolution for the coordinate ring of $\Obar$ by reducing to the case of $\OO'$.

In practice, many of the orbit closures of degenerate orbits that we encountered are well known varieties. When possible, we describe the defining equations directly. For some cases, we can obtain the defining ideal by polarizing the equations for the closure of the smaller orbit with respect to the inclusion $\gg'_1 \hookrightarrow \gg_1$, and adding (if they are not already contained in the ideal) the equations of the so-called generic degenerate orbit. The latter is an orbit closure in $\gg_1$ arising from an orbit closure in $\gg'_1$ whose closure is all of $\gg'_1$; it will be indicated on a case by case basis, when needed.
\section{Representations of type $E_6$}\label{type_E6}
In this section, we will analyze the cases corresponding to gradings on the simple Lie algebra of type $E_6$. Each case corresponds to the choice of a distinguished node on the Dynkin diagram for $E_6$. The nodes are numbered according to the conventions in Bourbaki \cite{MR0240238}.
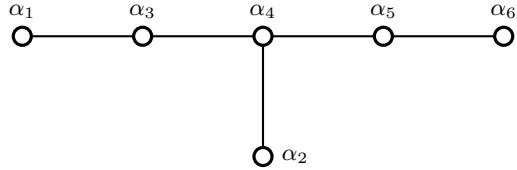
\begin{center}
\begin{tikzpicture}[scale=0.8]
	\singleline{(0,0)}
	\singleline{(2,0)}
	\singleline{(4,0)}
	\singleline{(6,0)}
	\singleverticalline{(4,0)}
	\rootnode{(0,0)}{$\alpha_1$}
	\rootnode[right]{(4,-2)}{$\alpha_2$}
	\rootnode{(2,0)}{$\alpha_3$}
	\rootnode{(4,0)}{$\alpha_4$}
	\rootnode{(6,0)}{$\alpha_5$}
	\rootnode{(8,0)}{$\alpha_6$}
\end{tikzpicture}
\end{center}

Given the symmetry in the diagram, it is enough to consider the cases with $\alpha_1, \alpha_2, \alpha_3, \alpha_4$ as the distinguished nodes. The case $(E_6,\alpha_5)$ is equivalent to $(E_6,\alpha_3)$, while $(E_6,\alpha_6)$ is equivalent to $(E_6,\alpha_1)$.

\subsection{The case $(E_{6}, \alpha_1)$}

The representation is $V(\omega_4, D_5)$, the half spinor representation for the group $\Spin (10)$. It has dimension 16 and weight vectors of the form $(\lambda_1, \ldots, \lambda_5)$, where $\lambda_i = \pm \frac{1}{2}$, with an even number of negative coordinates. Each weight is labeled by $[I]$, where
\[I=\left\{ i \in \{1,2,3,4,5\} \big| \lambda_i = -\tfrac{1}{2} \right\}.\]
The corresponding polynomial ring is
\[A=\CC [x_\varnothing, x_{ab}, x_{ijkl} | 1 \leq a < b \leq 5, 1 \leq i < j < k < l \leq 5].\]
In characteristic zero, the representation has the following orbits, listed along with the dimension of the closure and a representative:
\begin{center}
\begin{tabular}{ccc}
orbit & dimension & representative\\ \hline
$\mathcal{O}_0$ & 0& 0\\
$\mathcal{O}_1$ & 11& $x_\varnothing =1$\\
$\mathcal{O}_2$ & 16& $x_\varnothing = x_{1234} = 1$\\
\end{tabular}
\end{center}
All the orbit closures are normal, Cohen-Macaulay and have rational singularities. Here is the containment and singularity table:
\begin{center}
\begin{tabular}{c|c|c|c|}
& $\Obar_0$ & $\Obar_1$ & $\Obar_2$\\ \hline \hline
$\OO_0$ & ns & s & ns\\ \hline
$\OO_1$ & & ns & ns\\ \hline
$\OO_2$ & & & ns\\ \hline
\end{tabular}
\end{center}

\subsubsection{The orbit $\OO_1$}

The variety $\Obar_1$ is the closure of the highest weight vector orbit and it is known as the variety of pure spinors. The defining equations were described by Manivel in \cite{MR2529169} and can be resolved directly in M2. The Betti table for the resolution is
\[
\begin{matrix}
&0&1&2&3&4&5\\\text{total:}&1&10&16&16&10&1\\\text{0:}&1&\text{.}&\text{.}&\text{.}&\text{.}&\text{.}\\\text{1:}&\text{.}&10&16&\text{.}&\text{.}&\text{.}\\\text{2:}&\text{.}&\text{.}&\text{.}&16&10&\text{.}\\\text{3:}&\text{.}&\text{.}&\text{.}&\text{.}&\text{.}&1\\\end{matrix}
\]
This orbit closure is Gorenstein.

\subsection{The case $(E_{6}, \alpha_2)$}

The representation is $\bigwedge^3 F$, where $F=\CC^6$; the group acting is $\GL (F)$. The weights of $\bigwedge^3 F$ are of the form $\epsilon_i + \epsilon_j + \epsilon_k$ for $1 \leq i < j < k \leq 6$. We label the corresponding weight vector by $[ijk]$ where $1 \leq i < j < k \leq 6$. The corresponding polynomial ring is
\[A=\CC [x_{ijk} | 1 \leq i < j < k \leq 6]= \Sym \left( \textstyle\bigwedge^3 F^* \right).\]

In characteristic zero, the representation has the following orbits, listed along with the dimension of the closure and a representative:
\begin{center}
\begin{tabular}{ccc}
orbit & dimension & representative\\ \hline
$\OO_0$ & 0& 0\\
$\OO_1$ & 10& $x_{123} =1$\\
$\OO_2$ & 15& $x_{123} = x_{145} =1$\\
$\OO_3$ & 19& $x_{123} = x_{145} = x_{246} =1$\\
$\OO_4$ & 20& $x_{123} = x_{456} = 1$\\
\end{tabular}
\end{center}
All the orbit closures are normal, Cohen-Macaulay, Gorenstein and have rational singularities. Here is the containment and singularity table:
\begin{center}
\begin{tabular}{c|c|c|c|c|c|}
& $\Obar_0$ & $\Obar_1$ & $\Obar_2$ & $\Obar_3$ & $\Obar_4$\\ \hline \hline
$\OO_0$ & ns & s & s & s & ns\\ \hline
$\OO_1$ & & ns & s & s & ns\\ \hline
$\OO_2$ & & & ns & s & ns\\ \hline
$\OO_3$ & & & & ns & ns\\ \hline
$\OO_4$ & & & & & ns\\ \hline
\end{tabular}
\end{center}

When describing the resolutions over $A$ we write simply $\lambda$ for
the Schur module $\Sc_\lambda F^* \otimes A(-|\lambda|/3)$. 
\subsubsection{The orbit $\OO_3$}
\label{sec:orbit-oo_3-1}

The orbit closure $\Obar_3$ is a hypersurface defined by an invariant
of degree 4 which can be obtained as follows:
\begin{center}
  \begin{tikzpicture}
    \matrix (m) [matrix of math nodes, row sep=20pt, column sep=42pt,
    text height=1.5ex, text depth=0.25ex, ampersand replacement=\&]
    {\bigwedge^6 F^* \otimes \bigwedge^6 F^*\\
    \bigwedge^3 F^* \otimes \bigwedge^2 F^* \otimes F^* \otimes F^*
    \otimes \bigwedge^2 F^* \otimes \bigwedge^3 F^* \\
    \bigwedge^3 F^* \otimes \bigwedge^3 F^* \otimes \bigwedge^3 F^* \otimes \bigwedge^3 F^* \\
    A_4\\};
    \path[->,semithick,font=\scriptsize]
    (m-1-1) edge node[auto] {$\Delta \otimes \Delta$} (m-2-1)
    (m-2-1) edge node[auto] {$m_{2,4} \otimes m_{3,5}$} (m-3-1);
    \path[right hook->,semithick,font=\scriptsize]
    (m-3-1) edge node[auto] {} (m-4-1);
  \end{tikzpicture}
\end{center}
An explicit formula for this invariant can be found in \cite[Remark 4.2]{MR630524}.
\subsubsection{The orbit $\OO_2$}

The expected resolution for the coordinate ring $\CC [\Obar_2]$ is
\[ (0^6) \longleftarrow (2^3, 1^3) \longleftarrow (3, 2^4, 1) \longleftarrow (4, 3^4, 2) \longleftarrow (4^3, 3^3) \longleftarrow (5^6) \longleftarrow 0\]
The differential $d_2$ was written explicitly as the map
\begin{center}
  \begin{tikzpicture}
    \matrix (m) [matrix of math nodes, row sep=10pt, column sep=45pt,
    text height=1.5ex, text depth=0.25ex, ampersand replacement=\&]
    {F\otimes F^* \& F\otimes F^* \otimes \bigwedge^2 F \otimes
      \bigwedge^2 F^* \& \bigwedge^3 F \otimes A_1\\};
    \path[->,semithick,font=\scriptsize]
    (m-1-1) edge node[auto] {$\tr^{(2)}$} (m-1-2)
    (m-1-2) edge node[auto] {$m_{13} \otimes m_{24}$} (m-1-3);
  \end{tikzpicture}
\end{center}
restricted to the space of $6\times 6$ traceless matrices $\ker
(F\otimes F^* \rightarrow \CC)$ identified with $\Sc_{(2,1^4)}
F^*$. The Betti table for the resolution is
\[
\begin{matrix}
&0&1&2&3&4&5\\\text{total:}&1&20&35&35&20&1\\\text{0:}&1&\text{.}&\text{.}&\text{.}&\text{.}&\text{.}\\\text{1:}&\text{.}&\text{.}&\text{.}&\text{.}&\text{.}&\text{.}\\\text{2:}&\text{.}&20&35&\text{.}&\text{.}&\text{.}\\\text{3:}&\text{.}&\text{.}&\text{.}&35&20&\text{.}\\\text{4:}&\text{.}&\text{.}&\text{.}&\text{.}&\text{.}&\text{.}\\\text{5:}&\text{.}&\text{.}&\text{.}&\text{.}&\text{.}&1\\\end{matrix}
\]

\subsubsection{The orbit $\OO_1$}

The variety $\Obar_1$ is the closure of the highest weight vector orbit. Geometrically, it is the cone over the Grassmannian $\Gr (3, \CC^6)$. The defining ideal is generated by Pl\"ucker relations and can be obtained in M2 using {\tt Grassmannian(2, 5, CoefficientRing => QQ)}. The Betti table of the minimal free resolution is
\[\begin{matrix}
&0&1&2&3&4&5&6&7&8&9&10\\\text{total:}&1&35&140&301&735&1080&735&301&140&35&1\\\text{0:}&1&\text{.}&\text{.}&\text{.}&\text{.}&\text{.}&\text{.}&\text{.}&\text{.}&\text{.}&\text{.}\\\text{1:}&\text{.}&35&140&189&\text{.}&\text{.}&\text{.}&\text{.}&\text{.}&\text{.}&\text{.}\\\text{2:}&\text{.}&\text{.}&\text{.}&112&735&1080&735&112&\text{.}&\text{.}&\text{.}\\\text{3:}&\text{.}&\text{.}&\text{.}&\text{.}&\text{.}&\text{.}&\text{.}&189&140&35&\text{.}\\\text{4:}&\text{.}&\text{.}&\text{.}&\text{.}&\text{.}&\text{.}&\text{.}&\text{.}&\text{.}&\text{.}&1\\\end{matrix}
\]
This resolution was first determined by Pragacz and Weyman in \cite{MR926298}.

\subsection{The case $(E_6, \alpha_3)$}
\label{sec:case-e_6-alpha_3}

The representation is $E \otimes \bigwedge^2 F$, where $E=\CC^2$ and
$F=\CC^5$; the group acting is $\SL (E) \times \SL (F) \times
\CC^\times$. We denote the tensor $e_a \otimes f_i \wedge f_j$ by
$[a;ij]$ where $a=1,2$ and $1 \leq i < j \leq 5$. The corresponding polynomial ring is
\[A=\CC [x_{a;ij} | a=1,2; 1 \leq i < j \leq 5]= \Sym \left(
  \textstyle E^* \otimes\bigwedge^2 F^* \right).\]

In characteristic zero, the representation has the following orbits, listed along with the dimension of the closure and a representative:
\begin{center}
\begin{tabular}{ccc}
orbit & dimension & representative\\ \hline
$\OO_0$ & 0& 0\\
$\OO_1$ & 8& $x_{1;12} =1$\\
$\OO_2$ & 11& $x_{1;12} = x_{1;34} =1$\\
$\OO_3$ & 12& $x_{1;12} = x_{2;13} =1$\\
$\OO_4$ & 15& $x_{1;12} = x_{1;34} = x_{2;13} =1$\\
$\OO_5$ & 16& $x_{1;12} = x_{2;34} =1$\\
$\OO_6$ & 18& $x_{1;12} = x_{2;34} = x_{1;35} =1$\\
$\OO_7$ & 20& $x_{1;12} = x_{2;34} = x_{1;35} = x_{2;15} =1$
\end{tabular}
\end{center}
All the orbit closures, except for $\Obar_6$, are normal, Cohen-Macaulay and have rational singularities. Here is the containment and singularity table:
\begin{center}
\begin{tabular}{c|c|c|c|c|c|c|c|c|}
& $\Obar_0$ & $\Obar_1$ & $\Obar_2$ & $\Obar_3$ & $\Obar_4$ & $\Obar_5$ & $\Obar_6$ & $\Obar_7$\\ \hline \hline
$\OO_0$ & ns & s & s & s & s & s & s & ns\\ \hline
$\OO_1$ & & ns & ns & s & s & s & s & ns\\ \hline
$\OO_2$ & & & ns & & s & ns & s & ns\\ \hline
$\OO_3$ & & & & ns & s & s & s & ns\\ \hline
$\OO_4$ & & & & & ns & ns & s & ns\\ \hline
$\OO_5$ & & & & & & ns & s & ns\\ \hline
$\OO_6$ & & & & & & & ns & ns\\ \hline
$\OO_7$ & & & & & & & & ns\\ \hline
\end{tabular}
\end{center}

We denote the free $A$-module $\Sc_{(a,b)} E^* \otimes \Sc_{(c,d,e,f,g)} F^* \otimes A(-a-b)$ by $(a,b;c,d,e,f,g)$.

\subsubsection{The orbit $\OO_6$}
\label{sec:E63O6}

The orbit closure $\Obar_6$ is not normal. The expected 
resolution for the coordinate ring of the normalization of
$\Obar_6$ is
\[A \oplus (1,1;1,1,1,1,0) \leftarrow (2,1;2,1,1,1,1) \leftarrow
(4,1;2,2,2,2,2) \leftarrow 0\]
The differential $d_2 : (4,1;2,2,2,2,2) \rightarrow (2,1;2,1,1,1,1)$ was written explicitly as the map
\begin{center}
  \begin{tikzpicture}
    \matrix (m) [matrix of math nodes, row sep=10pt, column sep=42pt,
    text height=1.5ex, text depth=0.25ex, ampersand replacement=\&]
    {\Sc_3 E^* \& \Sc_2 E^* \otimes E^* \& \Sc_2 E^* \otimes E^* \otimes
      \bigwedge^4 F \otimes \bigwedge^4 F^*\\};
    \path[->,semithick,font=\scriptsize]
    (m-1-1) edge node[auto] {$\Delta$} (m-1-2)
    (m-1-2) edge node[auto] {$\tr^{(4)}$} (m-1-3);
  \end{tikzpicture}
\end{center}
then $\Sc_2 E^* \otimes \bigwedge^4 F^*$ is embedded in $A_2$ via the map:
\begin{center}
  \begin{tikzpicture}
    \matrix (m) [matrix of math nodes, row sep=20pt, column sep=35pt,
    text height=1.5ex, text depth=0.25ex, ampersand replacement=\&]
    {\Sc_2 E^* \otimes \bigwedge^4 F^* \\ \Sc_2 E^* \otimes
      \bigwedge^2 F^* \otimes \bigwedge^2 F^* \\ \Sc_2 E^* \otimes
      \Sc_2 (\bigwedge^2 F^*) \\ A_2\\};
    \path[->,semithick,font=\scriptsize]
    (m-1-1) edge node[auto] {$\Delta$} (m-2-1)
    (m-2-1) edge node[auto] {$m_{2,3}$} (m-3-1);
    \path[right hook->,semithick]
    (m-3-1) edge node[auto] {} (m-4-1);
  \end{tikzpicture}
\end{center}
to get a map $\Sc_3 E^* \rightarrow E^* \otimes \bigwedge^4 F \otimes
A_2$. The Betti table for the normalization is
\[\begin{matrix}
&0&1&2\\\text{total:}&6&10&4\\\text{0:}&1&\text{.}&\text{.}\\\text{1:}&\text{.}&\text{.}&\text{.}\\\text{2:}&5&10&\text{.}\\\text{3:}&\text{.}&\text{.}&4\\\end{matrix}\]
Dropping the row of degree 3 in the first differential we get the map
\[(2,1;2,1,1,1,1) \rightarrow (1,1;1,1,1,1,0)\]
This is a presentation for the cokernel $C(6)$ of the inclusion $\CC [\Obar_6]
\hookrightarrow \CC [\mathcal{N} (\Obar_6)]$, whose Betti table is
\[\begin{matrix}
&0&1&2&3&4\\\text{total:}&5&10&14&10&1\\\text{2:}&5&10&\text{.}&\text{.}&\text{.}\\\text{3:}&\text{.}&\text{.}&4&\text{.}&\text{.}\\\text{4:}&\text{.}&\text{.}&10&10&\text{.}\\\text{5:}&\text{.}&\text{.}&\text{.}&\text{.}&\text{.}\\\text{6:}&\text{.}&\text{.}&\text{.}&\text{.}&1\\\end{matrix}\]
By the cone procedure, we recover the resolution of $\CC [\Obar_6]$
which has the following Betti table
\[\begin{matrix}
&0&1&2&3\\\text{total:}&1&10&10&1\\\text{0:}&1&\text{.}&\text{.}&\text{.}\\\text{1:}&\text{.}&\text{.}&\text{.}&\text{.}\\\text{2:}&\text{.}&\text{.}&\text{.}&\text{.}\\\text{3:}&\text{.}&\text{.}&\text{.}&\text{.}\\\text{4:}&\text{.}&\text{.}&\text{.}&\text{.}\\\text{5:}&\text{.}&10&10&\text{.}\\\text{6:}&\text{.}&\text{.}&\text{.}&\text{.}\\\text{7:}&\text{.}&\text{.}&\text{.}&1\\\end{matrix}\]
We observe that $\Obar_6$ is not Cohen-Macaulay because it has
codimension 2 but its coordinate ring has projective dimension 3.

\subsubsection{The orbit $\OO_5$}
\label{sec:orbit-oo_5}

The orbit closure $\Obar_5$ is degenerate. The expected resolution for the coordinate ring $\CC[\Obar_5]$ is
\begin{gather*}
  A \leftarrow (2,1;2,1,1,1,1) \leftarrow (4,1;2,2,2,2,2) \oplus
  (2,2;2,2,2,1,1) \leftarrow\\
  \leftarrow (4,3;3,3,3,3,2) \leftarrow (4,4;4,3,3,3,3) \leftarrow 0
\end{gather*}
We construct explicitly the differential $d_4$ as follows
\begin{center}
  \begin{tikzpicture}
    \matrix (m) [matrix of math nodes, row sep=20pt, column sep=35pt,
    text height=1.5ex, text depth=0.25ex, ampersand replacement=\&]
    {F^* \\ F^* \otimes E \otimes E^* \otimes F \otimes F^* \\ E
      \otimes E^* \otimes F \otimes \bigwedge^2 F^* \\ E \otimes F
      \otimes A_1\\};
    \path[->,semithick,font=\scriptsize]
    (m-1-1) edge node[auto] {$\tr^{(1)} \otimes \tr^{(1)}$} (m-2-1)
    (m-2-1) edge node[auto] {$m_{1,5}$} (m-3-1);
    \path[right hook->,semithick,font=\scriptsize]
    (m-3-1) edge node[auto] {$m_{2,4}$} (m-4-1);
  \end{tikzpicture}
\end{center}
The Betti table for the resolution is
\[\begin{matrix}
&0&1&2&3&4\\\text{total:}&1&10&14&10&5\\\text{0:}&1&\text{.}&\text{.}&\text{.}&\text{.}\\\text{1:}&\text{.}&\text{.}&\text{.}&\text{.}&\text{.}\\\text{2:}&\text{.}&10&10&\text{.}&\text{.}\\\text{3:}&\text{.}&\text{.}&4&\text{.}&\text{.}\\\text{4:}&\text{.}&\text{.}&\text{.}&10&5\\\end{matrix}
\]

\subsubsection{The orbit $\OO_4$}
The orbit closure $\Obar_4$ is degenerate and comes from a smaller orbit which is a hypersurface of degree 4 in $E \otimes \bigwedge^2 (\CC^4) \hookrightarrow E \otimes \bigwedge^2 F$. This hypersurface is defined by the discriminant of the Pfaffian of a $4\times 4$ skew-symmetric matrix of generic linear forms in two variables. The equations of $\Obar_4$ in $E \otimes \bigwedge^2 F$ are obtained by taking polarizations of this discriminant with respect to the inclusion $E \otimes \bigwedge^2 \CC^4 \hookrightarrow E \otimes \bigwedge^2 F$ (corresponding to the representation $\Sc_{(2,2)} E^* \otimes \Sc_{(2,2,2,2)} F^*$) together with the defining  equations of the ``generic degenerate orbit'' $\Obar_5$. The Betti table for the resolution is
\[
\begin{matrix}
&0&1&2&3&4&5\\\text{total:}&1&25&62&55&20&3\\\text{0:}&1&\text{.}&\text{.}&\text{.}&\text{.}&\text{.}\\\text{1:}&\text{.}&\text{.}&\text{.}&\text{.}&\text{.}&\text{.}\\\text{2:}&\text{.}&10&10&\text{.}&\text{.}&\text{.}\\\text{3:}&\text{.}&15&52&45&\text{.}&\text{.}\\\text{4:}&\text{.}&\text{.}&\text{.}&10&20&\text{.}\\\text{5:}&\text{.}&\text{.}&\text{.}&\text{.}&\text{.}&3\\\end{matrix}
\]

\subsubsection{The orbit $\OO_3$}

The orbit closure $\Obar_3$ is degenerate. The equations can be obtained directly through the map
\begin{center}
  \begin{tikzpicture}
    \matrix (m) [matrix of math nodes, row sep=10pt, column sep=27pt,
    text height=1.5ex, text depth=0.25ex, ampersand replacement=\&]
    {\Sc_2 E^* \otimes \bigwedge^4 F^*
      \& A_2\\};
    \path[right hook->,semithick,font=\scriptsize]
    (m-1-1) edge node[auto] {} (m-1-2);
  \end{tikzpicture}
\end{center}
which is the embedding described in \ref{sec:E63O6}. The Betti table for the resolution is
\[
\begin{matrix}
&0&1&2&3&4&5&6&7&8\\\text{total:}&1&15&75&187&265&245&121&20&5\\\text{0:}&1&\text{.}&\text{.}&\text{.}&\text{.}&\text{.}&\text{.}&\text{.}&\text{.}\\\text{1:}&\text{.}&15&20&\text{.}&\text{.}&\text{.}&\text{.}&\text{.}&\text{.}\\\text{2:}&\text{.}&\text{.}&55&152&105&\text{.}&\text{.}&\text{.}&\text{.}\\\text{3:}&\text{.}&\text{.}&\text{.}&35&160&245&120&\text{.}&\text{.}\\\text{4:}&\text{.}&\text{.}&\text{.}&\text{.}&\text{.}&\text{.}&1&20&5\\\end{matrix}
\]

\subsubsection{The orbit $\OO_2$}

The orbit closure $\Obar_2$ is degenerate. The equations can be obtained directly through the map
\begin{center}
  \begin{tikzpicture}
    \matrix (m) [matrix of math nodes, row sep=10pt, column sep=27pt,
    text height=1.5ex, text depth=0.25ex, ampersand replacement=\&]
    {\bigwedge^2 E^* \otimes \bigwedge^2 \left( \bigwedge^2 F^* \right) \& A_2\\};
    \path[right hook->,semithick,font=\scriptsize]
    (m-1-1) edge node[auto] {} (m-1-2);
  \end{tikzpicture}
\end{center}
which gives the $2\times 2$ minors of the generic matrix of a linear map $E^* \rightarrow \bigwedge^2 F^*$. The ideal is resolved by the Eagon-Northcott complex with Betti table
\[
\begin{matrix}
&0&1&2&3&4&5&6&7&8&9\\\text{total:}&1&45&240&630&1008&1050&720&315&80&9\\\text{0:}&1&\text{.}&\text{.}&\text{.}&\text{.}&\text{.}&\text{.}&\text{.}&\text{.}&\text{.}\\\text{1:}&\text{.}&45&240&630&1008&1050&720&315&80&9\\\end{matrix}
\]

\subsubsection{The orbit $\OO_1$}
The orbit closure $\Obar_1$ is degenerate. The equations are simply those of the orbit closures $\Obar_2$ and $\Obar_3$ taken together. The Betti table of the resolution is
\[
\begin{matrix}
&0&1&2&3&4&5&6&7&8&9&10&11&12\\\text{total:}&1&60&360&1011&1958&3750&5490&5235&3257&1329&375&60&4\\\text{0:}&1&\text{.}&\text{.}&\text{.}&\text{.}&\text{.}&\text{.}&\text{.}&\text{.}&\text{.}&\text{.}&\text{.}&\text{.}\\\text{1:}&\text{.}&60&360&1005&1458&1050&720&315&80&9&\text{.}&\text{.}&\text{.}\\\text{2:}&\text{.}&\text{.}&\text{.}&6&500&2700&4770&4920&3177&1200&285&30&\text{.}\\\text{3:}&\text{.}&\text{.}&\text{.}&\text{.}&\text{.}&\text{.}&\text{.}&\text{.}&\text{.}&120&90&30&4\\\end{matrix}
\]

\subsection{The case $(E_6, \alpha_4)$}
The representation is $E\otimes F \otimes H$, where $E=\CC^2$ and $F=H=\CC^3$; the group acting is $\SL (E) \times \SL(F) \times \SL(H) \times \CC^\times$. We denote the tensor $e_i \otimes f_j \otimes h_k$ by [$i$;$j$;$k$], where $i=1,2$  and $j,k=1,2,3$. The corresponding polynomial ring is
\[A=\CC [x_{ijk} | i=1,2; j,k=1,2,3] = \Sym \left( E^* \otimes F^* \otimes H^* \right)\]

In characteristic zero, the representation has the following orbits, listed along with the dimension of the closure and a representative.
\begin{center}
\begin{tabular}{ccc}
orbit & dimension & representative\\ \hline
$\OO_0$ & 0 & 0 \\
$\OO_1$ & 6 & $x_{111} =1$\\
$\OO_2$ & 8 & $x_{111} = x_{221} =1$\\
$\OO_3$ & 8 & $x_{111} = x_{212} =1$\\
$\OO_4$ & 9 & $x_{111} = x_{122} =1$\\
$\OO_5$ & 11 & $x_{111} = x_{122} = x_{212} =1$\\
$\OO_6$ & 10 & $x_{111} = x_{122} = x_{133} =1$\\
$\OO_7$ & 12 & $x_{111} = x_{222} =1$\\
$\OO_8$ & 13 & $x_{111} = x_{222} = x_{132} =1$\\
$\OO_9$ & 13 & $x_{111} = x_{222} = x_{123} =1$\\
$\OO_{10}$ & 14 & $x_{111} = x_{222} = x_{132} = x_{231} =1$\\
$\OO_{11}$ & 14 & $x_{111} = x_{222} = x_{123} = x_{213} =1$\\
$\OO_{12}$ & 14 & $x_{111} = x_{222} = x_{123} = x_{132} =1$\\
$\OO_{13}$ & 14 & $x_{111} = x_{222} = x_{123} = x_{231} =1$\\
$\OO_{14}$ & 15 & $x_{111} = x_{222} = x_{133} =1$\\
$\OO_{15}$ & 16 & $x_{111} = x_{222} = x_{123} = x_{231} = x_{132} =1$\\
$\OO_{16}$ & 17 & $x_{111} = x_{222} = x_{133} = x_{213} =1$\\
$\OO_{17}$ & 18 & $x_{111} = x_{211} = x_{122} = x_{133} =1, x_{222}=-1$
\end{tabular}
\end{center}

The containment and singularity table can be found below.

\begin{sidewaystable}
\centering
\begin{tabular}{c|c|c|c|c|c|c|c|c|c|c|c|c|c|c|c|c|c|c|}
& $\Obar_0$ & $\Obar_1$ & $\Obar_2$ & $\Obar_3$ & $\Obar_4$ & $\Obar_5$ & $\Obar_6$ & $\Obar_7$ & $\Obar_8$ & $\Obar_9$ & $\Obar_{10}$ & $\Obar_{11}$ & $\Obar_{12}$ & $\Obar_{13}$ & $\Obar_{14}$ & $\Obar_{15}$ & $\Obar_{16}$ & $\Obar_{17}$\\ \hline \hline
    $\OO_{0}$ & ns & s & s & s & s & s & s & s & s & s & s & s & s & s & s & s & s & ns  \\ \hline
    $\OO_{1}$ & ¥ & ns & ns & ns & s & s & ns & s & s & s & s & s & s & s & s & s & s & ns  \\ \hline
    $\OO_{2}$ & ¥ & ¥ & ns & ¥ & ¥ & s & ¥ & s & s & s & s & ns & s & s & s & s & s & ns \\ \hline
    $\OO_{3}$ & ¥ & ¥ & ¥ & ns & ¥ & s & ¥ & s & s & s & ns & s & s & s & s & s & s & ns \\ \hline
    $\OO_{4}$ & ¥ & ¥ & ¥ & ¥ & ns & s & ns & ns & s & s & ns & ns & s & s & s & s & s & ns \\ \hline
    $\OO_{5}$ & ¥ & ¥ & ¥ & ¥ & ¥ & ns & ¥ & ns & s & s & ns & ns & s & s & s & s & s & ns \\ \hline
    $\OO_{6}$ & ¥ & ¥ & ¥ & ¥ & ¥ & ¥ & ns & ¥ & ¥ & ¥ & ¥ & ¥ & s & ¥ & s & s & s & ns \\ \hline
    $\OO_{7}$ & ¥ & ¥ & ¥ & ¥ & ¥ & ¥ & ¥ & ns & s & s & ns & ns & s & s & s & s & s & ns \\ \hline
    $\OO_{8}$ & ¥ & ¥ & ¥ & ¥ & ¥ & ¥ & ¥ & ¥ & ns & ¥ & ns & ¥ & ns & ns & ns & s & s & ns \\ \hline
    $\OO_{9}$ & ¥ & ¥ & ¥ & ¥ & ¥ & ¥ & ¥ & ¥ & ¥ & ns & ¥ & ns & ns & ns & ns & s & s & ns \\ \hline
    $\OO_{10}$ & ¥ & ¥ & ¥ & ¥ & ¥ & ¥ & ¥ & ¥ & ¥ & ¥ & ns & ¥ & ¥ & ¥ & ¥ & s & s & ns \\ \hline
    $\OO_{11}$ & ¥ & ¥ & ¥ & ¥ & ¥ & ¥ & ¥ & ¥ & ¥ & ¥ & ¥ & ns & ¥ & ¥ & ¥ & s & s & ns \\ \hline
    $\OO_{12}$ & ¥ & ¥ & ¥ & ¥ & ¥ & ¥ & ¥ & ¥ & ¥ & ¥ & ¥ & ¥ & ns & ¥ & ns & s & s & ns \\ \hline
    $\OO_{13}$ & ¥ & ¥ & ¥ & ¥ & ¥ & ¥ & ¥ & ¥ & ¥ & ¥ & ¥ & ¥ & ¥ & ns & ¥ & s & s & ns \\ \hline
    $\OO_{14}$ & ¥ & ¥ & ¥ & ¥ & ¥ & ¥ & ¥ & ¥ & ¥ & ¥ & ¥ & ¥ & ¥ & ¥ & ns & ¥ & s & ns \\ \hline
    $\OO_{15}$ & ¥ & ¥ & ¥ & ¥ & ¥ & ¥ & ¥ & ¥ & ¥ & ¥ & ¥ & ¥ & ¥ & ¥ & ¥ & ns & s & ns \\ \hline
    $\OO_{16}$ & ¥ & ¥ & ¥ & ¥ & ¥ & ¥ & ¥ & ¥ & ¥ & ¥ & ¥ & ¥ & ¥ & ¥ & ¥ & ¥ & ns & ns \\ \hline
    $\OO_{17}$ & ¥ & ¥ & ¥ & ¥ & ¥ & ¥ & ¥ & ¥ & ¥ & ¥ & ¥ & ¥ & ¥ & ¥ & ¥ & ¥ & ¥ & ns \\ \hline
\end{tabular}
\caption{Containment and singularity table for the case $(E_6,\alpha_4)$}
\end{sidewaystable}

We denote the free $A$-module $\Sc_{(a,b)} E^* \otimes \Sc_{(c,d,e)} F^* \otimes \Sc_{(f,g,h)} F^* \otimes A(-a-b)$ by $(a,b;c,d,e;f,g,h)$.

Certain pairs of orbit closures are isomorphic under the involution exchanging $F$ and $H$. This involution produces an automorphism of $A$ exchanging $x_{ijk}$ with $x_{ikj}$; this, in turn, induces isomorphisms of the coordinate rings and free resolutions. Because of this, it is enough to discuss only one case in each pair.

\subsubsection{The orbit $\OO_{16}$}\label{E64O16}
The orbit closure $\Obar_{16}$ is a hypersurface defined by the discriminant of the determinant of a generic $3\times 3$ matrix of linear forms in two variables, which is a homogeneous polynomial of degree 12. Explicitly:
	\begin{align*}
	\delta  & = \det
	\begin{pmatrix}
	u x_{111} + v x_{211} & u x_{112} + v x_{212} & u x_{113} + v x_{213}\\
	u x_{121} + v x_{221} & u x_{122} + v x_{222} & u x_{123} + v x_{223}\\
	u x_{131} + v x_{231} & u x_{132} + v x_{232} & u x_{133} + v x_{233}
	\end{pmatrix} =\\
	& = a_{3,0} u^3 + a_{2,1} u^2 v + a_{1,2} u v^2 + a_{0,3} v^3,
	\end{align*}
and
	\[\disc (\delta) = 27 a_{3,0}^2 a_{1,2}^2 + 4 a_{3,0} a_{1,2}^3 + 4 a_{2,1}^3 a_{0,3} - a_{2,1}^2 a_{1,2}^2 -18 a_{3,0} a_{2,1} a_{1,2} a_{0,3}.\]
The orbit closure $\Obar_{16}$ is not normal. The expected resolution for the coordinate ring of the normalization $\mathcal{N} (\Obar_{16})$ is
\[A \oplus (2,1;1,1,1;1,1,1) \leftarrow (4,2;2,2,2;2,2,2) \leftarrow 0\]
so there is only one differential $d_1$ with two blocks. The block $(4,2;2,2,2;2,2,2) \rightarrow (2,1;1,1,1;1,1,1)$ is the map
\begin{center}
	\begin{tikzpicture}
	\matrix (m) [matrix of math nodes, row sep=20pt, column sep=35pt,
	text height=1.5ex, text depth=0.25ex, ampersand replacement=\&]
	{\Sc_{2} E^* \otimes \bigwedge^3 F^* \otimes \bigwedge^3 H^*\\
	E \otimes E^* \otimes \Sc_{2} E^* \otimes \bigwedge^3 F^* \otimes \bigwedge^3 H^*\\
	E \otimes \Sc_{3} E^* \otimes \bigwedge^3 F^* \otimes \bigwedge^3 H^*\\
	E \otimes A_3\\};
	\path[->,semithick,font=\scriptsize]
	(m-1-1) edge node[auto] {$\tr^{(1)}$} (m-2-1)
	(m-2-1) edge node[auto] {$m_{2,3}$} (m-3-1);
	\path[right hook->,semithick]
	(m-3-1) edge node[auto] {} (m-4-1);
	\end{tikzpicture}
\end{center}
where the embedding $\Sc_{3} E^* \otimes \bigwedge^3 F^* \otimes \bigwedge^3 H^* \hookrightarrow A_3$ at the end is given by
\[(e^*_1)^i (e^*_2)^j \longmapsto \frac{i! j!}{3!} a_{i,j},\]
the $a_{i,j}$ being the coefficients of $\delta$ as above. Notice that this block also provides a minimal presentation of the cokernel $C(16)$ of the inclusion $\CC [\Obar_{16}]
\hookrightarrow \CC [\mathcal{N} (\Obar_{16})]$.
The block $(4,2;2,2,2;2,2,2) \rightarrow A$ can be obtained as follows. First construct the following map for the $E^*$ factor:
\begin{center}
	\begin{tikzpicture}
	\matrix (m) [matrix of math nodes, row sep=20pt, column sep=35pt,
	text height=1.5ex, text depth=0.25ex, ampersand replacement=\&]
	{\Sc_{2} E^* \otimes \bigwedge^2 E^* \otimes \bigwedge^2 E^*\\
	E^* \otimes E^* \otimes E^* \otimes E^* \otimes E^* \otimes E^*\\
	\Sc_{3} E^* \otimes \Sc_{3} E^*\\};
	\path[->,semithick,font=\scriptsize]
	(m-1-1) edge node[auto] {$\Delta \otimes \Delta \otimes \Delta$} (m-2-1)
	(m-2-1) edge node[auto] {$m_{1,3,5} \otimes m_{2,4,6}$} (m-3-1);
	\end{tikzpicture}
\end{center}
Then embed into $A$ via the map
\begin{center}
	\begin{tikzpicture}
	\matrix (m) [matrix of math nodes, row sep=20pt, column sep=35pt,
	text height=1.5ex, text depth=0.25ex, ampersand replacement=\&]
	{\Sc_{3} E^* \otimes \bigwedge^3 F^* \otimes \bigwedge^3 H^* \otimes \Sc_{3} E^* \otimes  \bigwedge^3 F^* \otimes \bigwedge^3 H^*\\
	A_3 \otimes A_3\\
	A_6\\};
	\path[->,semithick,font=\scriptsize]
	(m-1-1) edge node[auto] {} (m-2-1)
	(m-2-1) edge node[auto] {} (m-3-1);
	\end{tikzpicture}
\end{center}
where the first step uses the embedding described earlier twice and the second step is symmetric multiplication.

The Betti table for the resolution of the normalization is
\[\begin{matrix}
&0&1\\\text{total:}&3&3\\\text{0:}&1&\text{.}\\\text{1:}&\text{.}&\text{.}\\\text{2:}&\text{.}&\text{.}\\\text{3:}&2&\text{.}\\\text{4:}&\text{.}&\text{.}\\\text{5:}&\text{.}&3\\\end{matrix}\]
and the Betti table for the resolution of the cokernel $C(16)$ is
\[\begin{matrix}
&0&1&2\\\text{total:}&2&3&1\\\text{3:}&2&\text{.}&\text{.}\\\text{4:}&\text{.}&\text{.}&\text{.}\\\text{5:}&\text{.}&3&\text{.}\\\text{6:}&\text{.}&\text{.}&\text{.}\\\text{7:}&\text{.}&\text{.}&\text{.}\\\text{8:}&\text{.}&\text{.}&\text{.}\\\text{9:}&\text{.}&\text{.}&\text{.}\\\text{10:}&\text{.}&\text{.}&1\\\end{matrix}\]

\subsubsection{The orbit $\OO_{15}$}\label{E64O15}
The orbit closure $\Obar_{15}$ is normal with rational singularities. The expected resolution for the coordinate ring $\CC [\Obar_{15}]$ is
\[ A \longleftarrow (4,2;2,2,2;2,2,2) \longleftarrow (5,4;3,3,3;3,3,3) \longleftarrow 0\]
The differential $d_2$ was written explicitly by taking the map
\begin{center}
	\begin{tikzpicture}
	\matrix (m) [matrix of math nodes, row sep=20pt, column sep=35pt,
	text height=1.5ex, text depth=0.25ex, ampersand replacement=\&]
	{\bigwedge^2 E^* \otimes \bigwedge^2 E^* \otimes E^*\\
	E^* \otimes E^* \otimes E^* \otimes E^* \otimes E^*\\
	\Sc_{2} E^* \otimes \Sc_{2} E^* \otimes E^*\\
	\Sc_{3} E^* \otimes \Sc_{2} E^*\\};
	\path[->,semithick,font=\scriptsize]
	(m-1-1) edge node[auto] {$\Delta \otimes \Delta$} (m-2-1)
	(m-2-1) edge node[auto] {$m_{1,3} \otimes m_{2,4}$} (m-3-1)
	(m-3-1) edge node[auto] {$m_{1,3}$} (m-4-1);
	\end{tikzpicture}
\end{center}
on the $E^*$ factor and then embedding $\Sc_{3} E^* \otimes \bigwedge^3 F^* \otimes \bigwedge^3 H^*$ into $A_3$ as described in \ref{E64O16}. The Betti table for the resolution is
\[\begin{matrix}
&0&1&2\\\text{total:}&1&3&2\\\text{0:}&1&\text{.}&\text{.}\\\text{1:}&\text{.}&\text{.}&\text{.}\\\text{2:}&\text{.}&\text{.}&\text{.}\\\text{3:}&\text{.}&\text{.}&\text{.}\\\text{4:}&\text{.}&\text{.}&\text{.}\\\text{5:}&\text{.}&3&\text{.}\\\text{6:}&\text{.}&\text{.}&\text{.}\\\text{7:}&\text{.}&\text{.}&2\\\end{matrix}\]
We conclude that $\Obar_{15}$ is \CM{}.

\subsubsection{The orbit $\OO_{14}$}\label{sec:E64O14}

The orbit closure $\Obar_{14}$ is not normal. The expected
resolution for the coordinate ring of the normalization
$\mathcal{N} (\Obar_{14})$ is
\begin{gather*}
A \oplus (1, 1; 1, 1, 0; 1, 1, 0) \leftarrow (2, 1; 1, 1, 1; 2, 1, 0) \oplus (2, 1; 2, 1, 0; 1, 1, 1) \leftarrow\\
\leftarrow (3, 1; 2, 1, 1; 2, 1, 1) \leftarrow (5, 1; 2, 2, 2; 2, 2, 2) \leftarrow 0
\end{gather*}
The differential $d_3 : (5, 1; 2, 2, 2; 2, 2, 2) \rightarrow (3, 1; 2, 1, 1; 2, 1, 1)$ was written explicitly as the map
\begin{center}
  \begin{tikzpicture}
    \matrix (m) [matrix of math nodes, row sep=20pt, column sep=35pt,
	text height=1.5ex, text depth=0.25ex, ampersand replacement=\&]
    {\Sc_{4} E^* \\
    \Sc_{2} E^* \otimes \Sc_{2} E^* \otimes \bigwedge^2 F \otimes \bigwedge^2 F^* \otimes \bigwedge^2 H \otimes \bigwedge^2 H^*\\
    \Sc_{2} E^* \otimes \bigwedge^2 F \otimes \bigwedge^2 H \otimes A_2\\};
    \path[->,semithick,font=\scriptsize]
    (m-1-1) edge node[auto] {$\Delta \otimes \tr^{(2)} \otimes \tr^{(2)}$} (m-2-1)
    (m-2-1) edge node[auto] {} (m-3-1);
  \end{tikzpicture}
\end{center}
The embedding $\Sc_{2} E^* \otimes \bigwedge^2 F^* \otimes \bigwedge^2 H^* \hookrightarrow A_2$ works by sending the basis vector
\[(e_1^*)^i (e_2^*)^j \otimes f_a \wedge f_b \otimes h_c \wedge h_d\]
to the coefficient of $u^i v^j$ in the expansion of
\[\frac{i!j!}{2!} \det 
	\begin{pmatrix}
	u x_{1ac} + v x_{2ac} & u x_{1ad} + v x_{2ad} \\
	u x_{1bc} + v x_{2bc} & u x_{1bd} + v x_{2bd}
	\end{pmatrix}
\]
Notice that the minor above is the one corresponding to rows $a,b$ and columns $c,d$ in the matrix $\delta$ defined in \ref{E64O16}.
The Betti table for the normalization is
\[\begin{matrix}
&0&1&2&3\\\text{total:}&10&32&27&5\\\text{0:}&1&\text{.}&\text{.}&\text{.}\\\text{1:}&\text{.}&\text{.}&\text{.}&\text{.}\\\text{2:}&9&32&27&\text{.}\\\text{3:}&\text{.}&\text{.}&\text{.}&5\\\end{matrix}\]
Dropping the row of degree 3 in the first differential we get the map
\[(2, 1; 1, 1, 1; 2, 1, 0) \oplus (2, 1; 2, 1, 0; 1, 1, 1) \rightarrow (1, 1; 1, 1, 0; 1, 1, 0)\]
This is a presentation for the cokernel $C(14)$ of the inclusion $\CC [\Obar_{14}]
\hookrightarrow \CC [\mathcal{N} (\Obar_{14})]$; the Betti table for $C(14)$ is
\[\begin{matrix}
&0&1&2&3&4&5&6&7&8\\\text{total:}&9&32&131&347&477&372&181&54&7\\\text{2:}&9&32&27&\text{.}&\text{.}&\text{.}&\text{.}&\text{.}&\text{.}\\\text{3:}&\text{.}&\text{.}&\text{.}&5&\text{.}&\text{.}&\text{.}&\text{.}&\text{.}\\\text{4:}&\text{.}&\text{.}&104&342&477&372&180&54&7\\\text{5:}&\text{.}&\text{.}&\text{.}&\text{.}&\text{.}&\text{.}&\text{.}&\text{.}&\text{.}\\\text{6:}&\text{.}&\text{.}&\text{.}&\text{.}&\text{.}&\text{.}&1&\text{.}&\text{.}\\\end{matrix}\]
By the truncated cone procedure, we recover the resolution of $\CC [\Obar_{14}]$
which has the following Betti table
\[\begin{matrix}
&0&1&2&3&4&5&6&7\\\text{total:}&1&104&342&477&372&181&54&7\\\text{0:}&1&\text{.}&\text{.}&\text{.}&\text{.}&\text{.}&\text{.}&\text{.}\\\text{1:}&\text{.}&\text{.}&\text{.}&\text{.}&\text{.}&\text{.}&\text{.}&\text{.}\\\text{2:}&\text{.}&\text{.}&\text{.}&\text{.}&\text{.}&\text{.}&\text{.}&\text{.}\\\text{3:}&\text{.}&\text{.}&\text{.}&\text{.}&\text{.}&\text{.}&\text{.}&\text{.}\\\text{4:}&\text{.}&\text{.}&\text{.}&\text{.}&\text{.}&\text{.}&\text{.}&\text{.}\\\text{5:}&\text{.}&104&342&477&372&180&54&7\\\text{6:}&\text{.}&\text{.}&\text{.}&\text{.}&\text{.}&\text{.}&\text{.}&\text{.}\\\text{7:}&\text{.}&\text{.}&\text{.}&\text{.}&\text{.}&1&\text{.}&\text{.}\\\end{matrix}\]
We observe that $\Obar_{14}$ is not Cohen-Macaulay because it has
codimension 3 but its coordinate ring has projective dimension 7.

\subsubsection{The orbit $\OO_{13}$}\label{sec:E64O13}

The orbit closure $\Obar_{13}$ is not normal. The expected 
resolution for the coordinate ring of the normalization
$\mathcal{N} (\Obar_{13})$ is
\begin{gather*}
A \oplus (1, 1; 1, 1, 0; 1, 1, 0) \leftarrow (2, 1; 1, 1, 1; 1, 1, 1) \oplus\\
\oplus (2, 1; 1, 1, 1; 2, 1, 0) \oplus (2, 1; 2, 1, 0; 1, 1, 1) \oplus (3, 0; 1, 1, 1; 1, 1, 1) \leftarrow\\
\leftarrow (2, 2; 2, 1, 1; 2, 1, 1) \oplus (3, 1; 2, 1, 1; 2, 1, 1) \oplus (3,2;2,2,2;2,2,2) \leftarrow\\
\leftarrow (4, 3; 3, 2, 2; 3, 2, 2) \leftarrow (4, 4; 3, 3, 2; 3, 3, 2) \leftarrow 0
\end{gather*}
The differential $d_4 : (4, 4; 3, 3, 2; 3, 3, 2) \rightarrow (4, 3; 3, 2, 2; 3, 2, 2)$ was written explicitly as the map
\begin{center}
  \begin{tikzpicture}
    \matrix (m) [matrix of math nodes, row sep=20pt, column sep=35pt,
	text height=1.5ex, text depth=0.25ex, ampersand replacement=\&]
    {\bigwedge^2 F^* \otimes \bigwedge^2 H^* \\
    E \otimes E^* \otimes F^* \otimes F^* \otimes H^* \otimes H^*\\
    E \otimes F^* \otimes H^* \otimes A_1\\};
    \path[->,semithick,font=\scriptsize]
    (m-1-1) edge node[auto] {$\tr^{(1)} \otimes \Delta \otimes \Delta$} (m-2-1)
    (m-2-1) edge node[auto] {$m_{2,4,6}$} (m-3-1);
  \end{tikzpicture}
\end{center}
The Betti table for the normalization is
\[\begin{matrix}
&0&1&2&3&4\\\text{total:}&10&38&37&18&9\\\text{0:}&1&\text{.}&\text{.}&\text{.}&\text{.}\\\text{1:}&\text{.}&\text{.}&\text{.}&\text{.}&\text{.}\\\text{2:}&9&38&36&\text{.}&\text{.}\\\text{3:}&\text{.}&\text{.}&\text{.}&\text{.}&\text{.}\\\text{4:}&\text{.}&\text{.}&1&18&9\\\end{matrix}\]
Dropping the row of degree 3 in the first differential we get the map
\[(2, 1; 1, 1, 1; 1, 1, 1)\oplus (2, 1; 1, 1, 1; 2, 1, 0) \oplus (2, 1; 2, 1, 0; 1, 1, 1) \rightarrow (1, 1; 1, 1, 0; 1, 1, 0)\]
Notice how the representation $(3, 0; 1, 1, 1; 1, 1, 1)$ was also dropped from the domain since it does not map to $(1, 1; 1, 1, 0; 1, 1, 0)$. This is a presentation for the cokernel $C(13)$ of the inclusion $\CC [\Obar_{13}]
\hookrightarrow \CC [\mathcal{N} (\Obar_{13})]$, whose Betti table is
\[\begin{matrix}
&0&1&2&3&4&5&6\\\text{total:}&9&34&56&95&99&36&1\\\text{2:}&9&34&36&\text{.}&\text{.}&\text{.}&\text{.}\\\text{3:}&\text{.}&\text{.}&\text{.}&5&\text{.}&\text{.}&\text{.}\\\text{4:}&\text{.}&\text{.}&20&90&99&36&\text{.}\\\text{5:}&\text{.}&\text{.}&\text{.}&\text{.}&\text{.}&\text{.}&\text{.}\\\text{6:}&\text{.}&\text{.}&\text{.}&\text{.}&\text{.}&\text{.}&1\\\end{matrix}\]
By the truncated cone procedure, we recover the resolution of $\CC [\Obar_{13}]$
which has the following Betti table
\[\begin{matrix}
&0&1&2&3&4&5\\\text{total:}&1&24&78&90&36&1\\\text{0:}&1&\text{.}&\text{.}&\text{.}&\text{.}&\text{.}\\\text{1:}&\text{.}&\text{.}&\text{.}&\text{.}&\text{.}&\text{.}\\\text{2:}&\text{.}&4&\text{.}&\text{.}&\text{.}&\text{.}\\\text{3:}&\text{.}&\text{.}&\text{.}&\text{.}&\text{.}&\text{.}\\\text{4:}&\text{.}&\text{.}&6&\text{.}&\text{.}&\text{.}\\\text{5:}&\text{.}&20&72&90&36&\text{.}\\\text{6:}&\text{.}&\text{.}&\text{.}&\text{.}&\text{.}&\text{.}\\\text{7:}&\text{.}&\text{.}&\text{.}&\text{.}&\text{.}&1\\\end{matrix}\]
We observe that $\Obar_{13}$ is not Cohen-Macaulay because it has
codimension 4 but its coordinate ring has projective dimension 5.

\subsubsection{The orbit $\OO_{12}$}\label{sec:E64O12}
The orbit closure $\Obar_{12}$ is not normal. The expected 
resolution for the coordinate ring of the normalization
$\mathcal{N} (\Obar_{12})$ is
\begin{gather*}
A \oplus (1, 1; 1, 1, 0; 1, 1, 0) \leftarrow\\
\leftarrow (2, 1; 1, 1, 1; 1, 1, 1) \oplus (2, 1; 1, 1, 1; 2, 1, 0) \oplus (2, 1; 2, 1, 0; 1, 1, 1) \leftarrow\\
\leftarrow (3,1;2,1,1;2,1,1) \oplus (3,3;2,2,2;3,2,1)\oplus (3,3;3,2,1;2,2,2)\leftarrow\\
\leftarrow (4,3;3,2,2;3,2,2) \oplus (5,1;2,2,2;2,2,2) \leftarrow (6,3;3,3,3;3,3,3) \leftarrow 0
\end{gather*}
The differential $d_4 : (6,3;3,3,3;3,3,3) \rightarrow (4,3;3,2,2;3,2,2) \oplus (5,1;2,2,2;2,2,2)$ was written explicitly. The block $(6,3;3,3,3;3,3,3) \rightarrow (4,3;3,2,2;3,2,2)$ was constructed as the map
\begin{center}
  \begin{tikzpicture}
    \matrix (m) [matrix of math nodes, row sep=20pt, column sep=35pt,
	text height=1.5ex, text depth=0.25ex, ampersand replacement=\&]
    {\Sc_{3} E^* \otimes \bigwedge^3 F^* \otimes \bigwedge^3 H^* \\
    E^* \otimes \Sc_{2} E^* \otimes F^* \otimes \bigwedge^2 F^* \otimes H^* \otimes \bigwedge^2 H^*\\
    E^* \otimes F^* \otimes H^* \otimes A_2\\};
    \path[->,semithick,font=\scriptsize]
    (m-1-1) edge node[auto] {$\Delta \otimes \Delta \otimes \Delta$} (m-2-1)
    (m-2-1) edge node[auto] {$m_{2,4,6}$} (m-3-1);
  \end{tikzpicture}
\end{center}
where the embedding $\Sc_{2} E^* \otimes \bigwedge^2 F^* \otimes \bigwedge^2 H^* \hookrightarrow A_2$ is the one described in \ref{sec:E64O14}. The second block corresponding to the map $(6,3;3,3,3;3,3,3)\rightarrow (5,1;2,2,2;2,2,2)$ was constructed as the map
\begin{center}
  \begin{tikzpicture}
    \matrix (m) [matrix of math nodes, row sep=20pt, column sep=35pt,
	text height=1.5ex, text depth=0.25ex, ampersand replacement=\&]
    { \bigwedge^2 E^* \otimes \bigwedge^2 E^* \otimes \Sc_{3} E^* \\
    E^* \otimes E^* \otimes E^* \otimes E^* \otimes \Sc_{2} E^* \otimes E^*\\
    \Sc_{4} E^* \otimes \Sc_{3} E^*\\};
    \path[->,semithick,font=\scriptsize]
    (m-1-1) edge node[auto] {$\Delta \otimes \Delta \otimes \Delta$} (m-2-1)
    (m-2-1) edge node[auto] {$m_{1,3,5} \otimes m_{2,4,6}$} (m-3-1);
  \end{tikzpicture}
\end{center}
on the $E^*$ factor and then embedding $\Sc_{3} E^* \otimes \bigwedge^3 F^* \otimes \bigwedge^3 H^*$ into $A_3$ as described in \ref{E64O16}.
The Betti table for the normalization is
\[\begin{matrix}
&0&1&2&3&4\\\text{total:}&10&34&43&23&4\\\text{0:}&1&\text{.}&\text{.}&\text{.}&\text{.}\\\text{1:}&\text{.}&\text{.}&\text{.}&\text{.}&\text{.}\\\text{2:}&9&34&27&\text{.}&\text{.}\\\text{3:}&\text{.}&\text{.}&\text{.}&5&\text{.}\\\text{4:}&\text{.}&\text{.}&16&18&\text{.}\\\text{5:}&\text{.}&\text{.}&\text{.}&\text{.}&4\\\end{matrix}\]
Dropping the row of degree 3 in the first differential we get the map
\[(2, 1; 1, 1, 1; 1, 1, 1)\oplus (2, 1; 1, 1, 1; 2, 1, 0) \oplus (2, 1; 2, 1, 0; 1, 1, 1) \rightarrow (1, 1; 1, 1, 0; 1, 1, 0)\]
which is a presentation for the cokernel $C(12)$ of the inclusion $\CC [\Obar_{12}]
\hookrightarrow \CC [\mathcal{N} (\Obar_{12})]$. Notice that this is the same as the presentation of $C(13)$ whose Betti table is described in \ref{sec:E64O13}.
By the truncated cone procedure, we recover the resolution of $\CC [\Obar_{12}]$
which has the following Betti table
\[\begin{matrix}
&0&1&2&3&4&5\\\text{total:}&1&29&88&99&40&1\\\text{0:}&1&\text{.}&\text{.}&\text{.}&\text{.}&\text{.}\\\text{1:}&\text{.}&\text{.}&\text{.}&\text{.}&\text{.}&\text{.}\\\text{2:}&\text{.}&\text{.}&\text{.}&\text{.}&\text{.}&\text{.}\\\text{3:}&\text{.}&9&\text{.}&\text{.}&\text{.}&\text{.}\\\text{4:}&\text{.}&\text{.}&16&\text{.}&\text{.}&\text{.}\\\text{5:}&\text{.}&20&72&99&40&\text{.}\\\text{6:}&\text{.}&\text{.}&\text{.}&\text{.}&\text{.}&\text{.}\\\text{7:}&\text{.}&\text{.}&\text{.}&\text{.}&\text{.}&1\\\end{matrix}\]
We observe that $\Obar_{12}$ is not Cohen-Macaulay because it has codimension 4 but its coordinate ring has projective dimension 5.

\subsubsection{The orbits $\OO_{11}$ and $\OO_{10}$}\label{sec:E64O11}
We discuss the case of the orbit $\OO_{11}$ since $\OO_{10}$ is isomorphic under the involution exchanging $F$ and $H$. The orbit closure $\Obar_{11}$ is normal with rational singularities. It is $F$-degenerate and its equations are the $3\times 3$ minors of
\[
	\begin{pmatrix}
	x_{111} & x_{121} & x_{131}\\
	x_{112} & x_{122} & x_{132}\\
	x_{113} & x_{123} & x_{133}\\
	x_{211} & x_{221} & x_{231}\\
	x_{212} & x_{222} & x_{232}\\
	x_{213} & x_{223} & x_{233}
	\end{pmatrix}
\]
the generic matrix of a linear map $F\rightarrow E^* \otimes H^*$. As such $\Obar_{11}$ is a determinantal variety and its resolution is given by the Eagon-Northcott complex with the following Betti table
\[\begin{matrix}
&0&1&2&3&4\\\text{total:}&1&20&45&36&10\\\text{0:}&1&\text{.}&\text{.}&\text{.}&\text{.}\\\text{1:}&\text{.}&\text{.}&\text{.}&\text{.}&\text{.}\\\text{2:}&\text{.}&20&45&36&10\\\end{matrix}\]
It follows that $\Obar_{11}$ is Cohen-Macaulay.

\subsubsection{The orbits $\OO_9$ and $\OO_8$}\label{sec:E64O8}
We discuss the case of the orbit $\OO_8$ since $\OO_9$ is isomorphic under the involution exchanging $F$ and $H$. The orbit closure $\Obar_8$ is normal with rational singularities. It is degenerate and comes from a smaller orbit which is a hypersurface of degree 6 in the representation
\[\Sc_{(3,3)} E^* \otimes \Sc_{(2,2,2)} F^* \otimes \Sc_{(3,3)} (\CC^2)^* \hookrightarrow \Sc_{(3,3)} E^* \otimes \Sc_{(2,2,2)} F^* \otimes \Sc_{(3,3)} H^*.\]
The hypersurface is defined by the invariant of degree 6 in $\CC^2 \otimes \CC^3 \otimes \CC^2$ which is the hyperdeterminant of the boundary format $2\times 3 \times 2$ (see \cite{MR1264417}). The equations of $\Obar_8$ are obtained by taking polarizations of such invariant with respect to the inclusion above together with the equations of the ``generic degenerate orbit'' $\Obar_{10}$. The Betti table for the resolution is
\[\begin{matrix}
&0&1&2&3&4&5\\\text{total:}&1&30&81&81&30&1\\\text{0:}&1&\text{.}&\text{.}&\text{.}&\text{.}&\text{.}\\\text{1:}&\text{.}&\text{.}&\text{.}&\text{.}&\text{.}&\text{.}\\\text{2:}&\text{.}&20&45&36&10&\text{.}\\\text{3:}&\text{.}&\text{.}&\text{.}&\text{.}&\text{.}&\text{.}\\\text{4:}&\text{.}&\text{.}&\text{.}&\text{.}&\text{.}&\text{.}\\\text{5:}&\text{.}&10&36&45&20&\text{.}\\\text{6:}&\text{.}&\text{.}&\text{.}&\text{.}&\text{.}&\text{.}\\\text{7:}&\text{.}&\text{.}&\text{.}&\text{.}&\text{.}&1\\\end{matrix}\]
It follows that $\Obar_8$ is Gorenstein.

\subsubsection{The orbit $\OO_7$}\label{sec:E64O7}
The orbit closure $\Obar_7$ is normal with rational singularities. It is $F$-$H$-degenerate with equations given by the $3\times 3$ minors of the generic matrix of a linear map $F\rightarrow E^* \otimes H^*$ together with the $3\times 3$ minors of the generic matrix of a linear map $H\rightarrow E^* \otimes F^*$; in other words, these are the equations of $\Obar_{11}$ and $\Obar_{10}$ taken together. The Betti table for the resolution is
\[\begin{matrix}
&0&1&2&3&4&5&6\\\text{total:}&1&36&99&95&56&34&9\\\text{0:}&1&\text{.}&\text{.}&\text{.}&\text{.}&\text{.}&\text{.}\\\text{1:}&\text{.}&\text{.}&\text{.}&\text{.}&\text{.}&\text{.}&\text{.}\\\text{2:}&\text{.}&36&99&90&20&\text{.}&\text{.}\\\text{3:}&\text{.}&\text{.}&\text{.}&5&\text{.}&\text{.}&\text{.}\\\text{4:}&\text{.}&\text{.}&\text{.}&\text{.}&36&34&9\\\end{matrix}\]
It follows that $\Obar_7$ is Cohen-Macaulay.

\subsubsection{The orbit $\OO_6$}\label{sec:E64O6}
The orbit closure $\Obar_6$ is normal with rational singularities. It is $E$-degenerate and its equations are the $2\times 2$ minors of
\[
	\begin{pmatrix}
	x_{111} & x_{211} \\
	x_{112} & x_{212} \\
	x_{113} & x_{213} \\
	x_{121} & x_{221} \\
	x_{122} & x_{222} \\
	x_{123} & x_{223} \\
	x_{131} & x_{231} \\
	x_{132} & x_{232} \\
	x_{133} & x_{233} 
	\end{pmatrix}
\]
the generic matrix of a linear map $E\rightarrow F^* \otimes H^*$. Therefore $\Obar_6$ is a determinantal variety and its resolution is given by the Eagon-Northcott complex with the following Betti table
\[\begin{matrix}
&0&1&2&3&4&5&6&7&8\\\text{total:}&1&36&168&378&504&420&216&63&8\\\text{0:}&1&\text{.}&\text{.}&\text{.}&\text{.}&\text{.}&\text{.}&\text{.}&\text{.}\\\text{1:}&\text{.}&36&168&378&504&420&216&63&8\\\end{matrix}\]
It follows that $\Obar_6$ is Cohen-Macaulay.

\subsubsection{The orbit $\OO_5$}\label{sec:E64O5}
The orbit closure $\Obar_5$ is normal with rational singularities. It is degenerate and comes from a smaller orbit which is a hypersurface of degree 4 in the representation
\[\Sc_{(2,2)} E^* \otimes \Sc_{(2,2)} (\CC^2)^* \otimes \Sc_{(2,2)} (\CC^2)^* \hookrightarrow \Sc_{(2,2)} E^* \otimes \Sc_{(2,2)} F^* \otimes \Sc_{(2,2)} H^*.\]
The hypersurface is defined by the invariant of degree 4 in $\CC^2 \otimes \CC^2 \otimes \CC^2$ which can be written as the discriminant of the determinant of a generic $2\times 2$ matrix of linear forms in two variables. Explicitly take:
\[\det
	\begin{pmatrix}
	u x_{111} + v x_{211} & u x_{112} + v x_{212} \\
	u x_{121} + v x_{221} & u x_{122} + v x_{222}
	\end{pmatrix}
=a_{2,0} u^2 + a_{1,1} uv + a_{0,2} v^2\]
and the discriminant is $4 a_{2,0} a_{0,2} - a_{1,1}^2$.
The equations of $\Obar_5$ are obtained by taking polarizations of such invariant with respect to the inclusion above together with the equations of the ``generic degenerate orbit'' $\Obar_7$. The Betti table for the resolution is
\[\begin{matrix}
&0&1&2&3&4&5&6&7\\\text{total:}&1&72&297&530&488&223&42&3\\\text{0:}&1&\text{.}&\text{.}&\text{.}&\text{.}&\text{.}&\text{.}&\text{.}\\\text{1:}&\text{.}&\text{.}&\text{.}&\text{.}&\text{.}&\text{.}&\text{.}&\text{.}\\\text{2:}&\text{.}&36&99&90&20&\text{.}&\text{.}&\text{.}\\\text{3:}&\text{.}&36&198&440&468&189&6&\text{.}\\\text{4:}&\text{.}&\text{.}&\text{.}&\text{.}&\text{.}&34&36&\text{.}\\\text{5:}&\text{.}&\text{.}&\text{.}&\text{.}&\text{.}&\text{.}&\text{.}&3\\\end{matrix}\]
It follows that $\Obar_5$ is \CM{}.

\subsubsection{The orbit $\OO_4$}\label{sec:E64O4}
The orbit closure $\Obar_4$ is normal with rational singularities. It is both $E$-degenerate and $F$-$H$-degenerate with equations given by the $2\times 2$ minors of the generic matrix of a linear map $E\rightarrow F^* \otimes H^*$ together with the coefficients of the determinant of a generic $3\times 3$ matrix of linear forms in two variables.
More explicitly, the former are the equations of $\Obar_6$ while the latter are the coefficients $a_{3,0}, a_{2,1}, a_{1,2}, a_{0,3}$ of $\det (\delta)$ as introduced in \ref{E64O16}. The Betti table for the resolution is
\[\begin{matrix}
&0&1&2&3&4&5&6&7&8&9\\\text{total:}&1&40&195&450&588&546&384&171&44&5\\\text{0:}&1&\text{.}&\text{.}&\text{.}&\text{.}&\text{.}&\text{.}&\text{.}&\text{.}&\text{.}\\\text{1:}&\text{.}&36&168&378&504&420&216&63&8&\text{.}\\\text{2:}&\text{.}&4&27&72&84&\text{.}&\text{.}&\text{.}&\text{.}&\text{.}\\\text{3:}&\text{.}&\text{.}&\text{.}&\text{.}&\text{.}&126&168&108&36&5\\\end{matrix}\]
It follows that $\Obar_4$ is Cohen-Macaulay.

\subsubsection{The orbits $\OO_3$ and $\OO_2$}\label{sec:E64O3}
We discuss the case of the orbit $\OO_3$ since $\OO_2$ is isomorphic under the involution exchanging $F$ and $H$. The orbit closure $\Obar_3$ is normal with rational singularities. It is $F$-degenerate and its equations are the $2\times 2$ minors of the generic matrix of a linear map $F\rightarrow E^* \otimes H^*$ (see \ref{sec:E64O11}). As such $\Obar_3$ is a determinantal variety and the coordinate ring is resolved by Lascoux's resolution \cite{MR1988690}. Here is the Betti table for the resolution
\[\begin{matrix}
&0&1&2&3&4&5&6&7&8&9&10\\\text{total:}&1&45&230&540&823&1015&1035&760&351&90&10\\\text{0:}&1&\text{.}&\text{.}&\text{.}&\text{.}&\text{.}&\text{.}&\text{.}&\text{.}&\text{.}&\text{.}\\\text{1:}&\text{.}&45&230&540&648&385&90&\text{.}&\text{.}&\text{.}&\text{.}\\\text{2:}&\text{.}&\text{.}&\text{.}&\text{.}&175&630&945&760&351&90&10\\\end{matrix}\]
It follows that $\Obar_3$ is Cohen-Macaulay.

\subsubsection{The orbit $\OO_1$}\label{sec:E64O1}
The orbit closure $\Obar_1$ is normal with rational singularities. It is both $E$-degenerate and $F$-$H$-degenerate with equations given by the $2\times 2$ minors of the generic matrix of a linear map $E\rightarrow F^* \otimes H^*$ together with the coefficients of the $2\times 2$ minors of a generic $3\times 3$ matrix of linear forms in two variables.
More explicitly, the former are the equations of $\Obar_6$ while the latter are the coefficients of the $2\times 2$ minors of $\delta$ as introduced in \ref{E64O16}. The Betti table for the resolution is
\[\begin{matrix}
&0&1&2&3&4&5&6&7&8&9&10&11&12\\\text{total:}&1&63&394&1179&2087&2692&3726&4383&3275&1530&407&45&2\\\text{0:}&1&\text{.}&\text{.}&\text{.}&\text{.}&\text{.}&\text{.}&\text{.}&\text{.}&\text{.}&\text{.}&\text{.}&\text{.}\\\text{1:}&\text{.}&63&394&1179&1980&1702&396&63&8&\text{.}&\text{.}&\text{.}&\text{.}\\\text{2:}&\text{.}&\text{.}&\text{.}&\text{.}&107&990&3330&4320&3267&1530&407&36&\text{.}\\\text{3:}&\text{.}&\text{.}&\text{.}&\text{.}&\text{.}&\text{.}&\text{.}&\text{.}&\text{.}&\text{.}&\text{.}&9&2\\\end{matrix}\]
It follows that $\Obar_1$ is Cohen-Macaulay.
\section{Representations of type $F_4$}\label{type_F4}
In this section, we will analyze the cases corresponding to gradings on the simple Lie algebra of type $F_4$. Each case corresponds to the choice of a distinguished node on the Dynkin diagram for $F_4$. The nodes are numbered according to the conventions in Bourbaki \cite{MR0240238}.
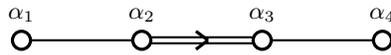
\begin{center}
\begin{tikzpicture}[scale=0.8]
	\singleline{(0,0)}
	\doublelinearrowright{(2,0)}
	\singleline{(4,0)}
	\rootnode{(0,0)}{$\alpha_1$}
	\rootnode{(2,0)}{$\alpha_2$}
	\rootnode{(4,0)}{$\alpha_3$}
	\rootnode{(6,0)}{$\alpha_4$}
\end{tikzpicture}
\end{center}

\subsection{The case $(F_{4}, \alpha_1)$}\label{F41}

The representation is $V=V(\omega_3,C_3)$, the third fundamental representation of the group $\Sp (6,\CC)$; the group acting is $\CC^\times \times \Sp (6,\CC)$. The corresponding polynomial ring is
\begin{align*}
A=\CC [ & x_{1342},x_{1242},x_{1232},x_{1231},x_{1222},x_{1221},x_{1122},\\
& x_{1220},x_{1121},x_{1120},x_{1111},x_{1110},x_{1100},x_{1000}] = \Sym (V^*).
\end{align*}
The variables in $A$ are indexed by the roots in $\gg_1 = V$. The variables are weight vectors in $V^*$, with the following weights:
\[
\begingroup
\renewcommand*{\arraystretch}{1.1}
\begin{array}{lcl}
x_{1342} \leftrightarrow \epsilon_1 + \epsilon_2 + \epsilon_3 & & x_{1242} \leftrightarrow \epsilon_1 + \epsilon_2 - \epsilon_3\\
x_{1232} \leftrightarrow \epsilon_1 & & x_{1231} \leftrightarrow \epsilon_2\\
x_{1222} \leftrightarrow \epsilon_1 - \epsilon_2 + \epsilon_3 & & x_{1221} \leftrightarrow \epsilon_3\\
x_{1122} \leftrightarrow \epsilon_1 - \epsilon_2 - \epsilon_3 & & x_{1220} \leftrightarrow - \epsilon_1 + \epsilon_2 + \epsilon_3\\
x_{1121} \leftrightarrow - \epsilon_3 & & x_{1120} \leftrightarrow - \epsilon_1 + \epsilon_2 - \epsilon_3\\
x_{1111} \leftrightarrow - \epsilon_2 & & x_{1110} \leftrightarrow - \epsilon_1\\
x_{1100} \leftrightarrow - \epsilon_1 - \epsilon_2 + \epsilon_3 & & x_{1000} \leftrightarrow - \epsilon_1 - \epsilon_2 - \epsilon_3
\end{array}
\endgroup
\]
Here $\epsilon_1,\epsilon_2,\epsilon_3$ are the vectors in the coordinate basis of $\mathbb{R}^3$ (see \cite[\S12.1]{Humphreys:1978fk} for more information).

In characteristic zero, the representation has the following orbits, listed along with the dimension of the closure and a representative:
\begin{center}
\begin{tabular}{ccc}
orbit & dimension & representative\\ \hline
$\OO_0$ & 0& 0\\
$\OO_1$ & 7& $x_{1000}=1$\\
$\OO_2$ & 10& $x_{1111}=1$\\
$\OO_3$ & 13& $x_{1000}=x_{1231}=1$\\
$\OO_4$ & 14& $x_{1000}=x_{1342}=1$
\end{tabular}
\end{center}

The orbit closures, in particular their defining equations, have been discussed in \cite{MR2152707}. The treatment we give here follows a similar approach as the one used in the previous sections.
All the orbit closures are normal, Cohen-Macaulay, and have rational singularities. Here is the containment and singularity table:
\begin{center}
\begin{tabular}{c|c|c|c|c|c|}
& $\Obar_0$ & $\Obar_1$ & $\Obar_2$ & $\Obar_3$ & $\Obar_4$\\ \hline \hline
$\OO_0$ & ns & s & s & s & ns\\ \hline
$\OO_1$ & & ns & s & s & ns\\ \hline
$\OO_2$ & & & ns & s & ns\\ \hline
$\OO_3$ & & & & ns & ns\\ \hline
$\OO_4$ & & & & & ns\\ \hline
\end{tabular}
\end{center}

We will denote by $V_{\omega}$ the highest weight module with highest weight $\omega$. The weights will be expressed as linear combinations of $\omega_1$, $\omega_2$, $\omega_3$, the fundamental weights of the root system $C_3$.

Let $F=V_{\omega_1} = \CC^6$ be the standard representation of $\Sp (6,\CC)$. Since $\bigwedge^3 F \cong F \oplus V$, there is a projection $\phi : \bigwedge^3 F \rightarrow V$ of $\Sp (6,\CC)$-modules. Take $\psi : V \rightarrow \bigwedge^3 F$, to be a section of $\phi$, so $\phi\circ \psi = \id$. Also, let $\delta : F \rightarrow F^*$ be the duality given by the symplectic form on $F$.

Next we observe that $A_2 = \Sc_2 V^* \cong \Sc_2 F^* \oplus V_{2\omega_2}^*$. Therefore there exists a non zero map $\rho : \Sc_2 F^* \rightarrow A_2$ of $\Sp (6,\CC)$-modules. We explain here one possible way to write such a map explicitly, in terms of well understood equivariant maps:
\begin{center}
  \begin{tikzpicture}
    \matrix (m) [matrix of math nodes, row sep=20pt, column sep=42pt,
    text height=1.5ex, text depth=0.25ex, ampersand replacement=\&]
    {\Sc_2 F^*\\
    F^* \otimes F^*\\
    F \otimes F^*\\
    \bigwedge^5 F^* \otimes F^*\\
    \bigwedge^3 F^* \otimes \bigwedge^2 F^* \otimes F^*\\
    \bigwedge^3 F^* \otimes \bigwedge^3 F^*\\
    V^* \otimes V^*\\
    \Sc_2 V^*\\};
    \path[->,semithick,font=\scriptsize]
    (m-1-1) edge node[auto] {$\Delta$} (m-2-1)
    (m-2-1) edge node[auto] {$\delta^{-1}$} (m-3-1)
    (m-3-1) edge node[auto] {${}^*$} (m-4-1)
    (m-4-1) edge node[auto] {$\Delta$} (m-5-1)
    (m-5-1) edge node[auto] {$m_{2,3}$} (m-6-1)
    (m-6-1) edge node[auto] {$\psi^* \otimes \psi^*$} (m-7-1)
    (m-7-1) edge node[auto] {$m_{1,2}$} (m-8-1);
  \end{tikzpicture}
\end{center}
This map will be used in the description of the orbit closures for the case $(F_4,\alpha_1)$.

\subsubsection{The orbit $\OO_3$}
The orbit closure $\Obar_3$ is a hypersurface defined by an invariant of degree 4 which can be obtained as follows:
\begin{center}
  \begin{tikzpicture}
    \matrix (m) [matrix of math nodes, row sep=20pt, column sep=42pt,
    text height=1.5ex, text depth=0.25ex, ampersand replacement=\&]
    {\CC\\
    \Sc_2 F \otimes \Sc_2 F^*\\
    \Sc_2 F^* \otimes \Sc_2 F^*\\
    A_2 \otimes A_2\\
    A_4\\};
    \path[->,semithick,font=\scriptsize]
    (m-1-1) edge node[auto] {$\tr^{(2)}$} (m-2-1)
    (m-2-1) edge node[auto] {$\Sc_2 (\delta)$} (m-3-1)
    (m-3-1) edge node[auto] {$\rho\otimes \rho$} (m-4-1)
    (m-4-1) edge node[auto] {} (m-5-1);
  \end{tikzpicture}
\end{center}
where the last map is symmetric multiplication. Another description of this invariant was given by Landsberg and Manivel \cite{MR1832903}.

\subsubsection{The orbit $\OO_2$}
The expected resolution for the coordinate ring $\CC [\Obar_2]$ is
\begin{gather*}
	A \leftarrow V_{\omega_3} \otimes A(-3) \leftarrow V_{2\omega_1} \otimes A(-4) \leftarrow\\
	\leftarrow V_{\omega_2} \otimes A(-6) \leftarrow V_{\omega_1} \otimes A(-7) \leftarrow 0
\end{gather*}
The differential $d_2 : V_{2\omega_1} \otimes A(-4) \rightarrow V_{\omega_3} \otimes A(-3)$ was written explicitly, as follows:
\begin{center}
  \begin{tikzpicture}
    \matrix (m) [matrix of math nodes, row sep=20pt, column sep=42pt,
    text height=1.5ex, text depth=0.25ex, ampersand replacement=\&]
    {\Sc_2 F^*\\
    F^* \otimes F^*\\
    F \otimes F^*\\
    \bigwedge^5 F^* \otimes F^*\\
    \bigwedge^3 F^* \otimes \bigwedge^2 F^* \otimes F^*\\
    \bigwedge^3 F^* \otimes \bigwedge^3 F^*\\
    V^* \otimes A_1\\};
    \path[->,semithick,font=\scriptsize]
    (m-1-1) edge node[auto] {$\Delta$} (m-2-1)
    (m-2-1) edge node[auto] {$\delta^{-1}$} (m-3-1)
    (m-3-1) edge node[auto] {${}^*$} (m-4-1)
    (m-4-1) edge node[auto] {$\Delta$} (m-5-1)
    (m-5-1) edge node[auto] {$m_{2,3}$} (m-6-1)
    (m-6-1) edge node[auto] {$\psi^* \otimes \psi^*$} (m-7-1);
  \end{tikzpicture}
\end{center}
The Betti table for the resolution is
\[\begin{matrix}
     &0&1&2&3&4\\
     \text{total:}&1&14&21&14&6\\\text{0:}&1&\text{.}&\text{.}&\text{.}&\text{.}\\
     \text{1:}&\text{.}&\text{.}&\text{.}&\text{.}&\text{.}\\
     \text{2:}&\text{.}&14&21&\text{.}&\text{.}\\
     \text{3:}&\text{.}&\text{.}&\text{.}&14&6\\
     \end{matrix}
\]

\subsubsection{The orbit $\OO_1$}
The expected resolution for the coordinate ring $\CC [\Obar_2]$ is
\begin{gather*}
	A \leftarrow V_{2\omega_1} \otimes A(-2) \leftarrow V_{\omega_1 +\omega_2} \otimes A(-3) \leftarrow V_{\omega_1 +\omega_3} \otimes A(-4) \leftarrow\\
	\leftarrow V_{\omega_1 +\omega_3} \otimes A(-6) \leftarrow V_{\omega_1 +\omega_2} \otimes A(-7) \leftarrow V_{2\omega_1} \otimes A(-8) \leftarrow A(-10) \leftarrow 0
\end{gather*}
The first differential is precisely the map $\rho$ described in the introduction to this case. The Betti table for the resolution is
\[
\begin{matrix}
&0&1&2&3&4&5&6&7\\
\text{total:}&1&21&64&70&70&64&21&1\\
\text{0:}&1&\text{.}&\text{.}&\text{.}&\text{.}&\text{.}&\text{.}&\text{.}\\
\text{1:}&\text{.}&21&64&70&\text{.}&\text{.}&\text{.}&\text{.}\\
\text{2:}&\text{.}&\text{.}&\text{.}&\text{.}&70&64&21&\text{.}\\
\text{3:}&\text{.}&\text{.}&\text{.}&\text{.}&\text{.}&\text{.}&\text{.}&1\\
\end{matrix}
\]
It follows that the orbit closure $\Obar_1$ is Gorenstein.

\subsection{The case $(F_{4}, \alpha_2)$}
\label{F42}

The representation is $E \otimes \Sc_{2} F$, where $E=\CC^2$ and
$F=\CC^3$; the group acting is $\SL (E) \times \SL (F) \times
\CC^\times$. The corresponding polynomial ring is
\[A=\CC [x_{i;jk} | i=1,2; 1 \leq j \leq k \leq 3]= \Sym \left(
  \textstyle E^* \otimes\Sc_{2} F^* \right).\]

In characteristic zero, the representation has the following orbits, listed along with the dimension of the closure and a representative:
\begin{center}
\begin{tabular}{ccc}
orbit & dimension & representative\\ \hline
$\OO_0$ & 0& 0\\
$\OO_1$ & 4& $x_{1;11} = 1$\\
$\OO_2$ & 6& $x_{1;12} = 1$\\
$\OO_3$ & 7& $x_{1;11} = x_{1;23} = 1$\\
$\OO_4$ & 8& $x_{1;11} = x_{2;22} = 1$\\
$\OO_5$ & 8& $x_{1;23} = x_{2;13} = 1$\\
$\OO_6$ & 9& $x_{1;11} = x_{2;22} = x_{123} = 1$\\
$\OO_7$ & 10& $x_{1;11} = x_{2;23} = 1$\\
$\OO_8$ & 10& $x_{2;13} = x_{1;23} = x_{1;11} = 1$\\
$\OO_9$ & 11& $\langle x_{1;11}, x_{1;12}, x_{2;13}, x_{1;22}, x_{2;23}\rangle$\\
$\OO_{10}$ & 12& $x_{1;11} = x_{2;11} = x_{1;22} =  x_{1;33} = 1, x_{2;22} = -1$
\end{tabular}
\end{center}
To obtain a representative for the orbit $\OO_9$ it is enough to assign random rational values to the variables listed, and to set all other variables to 0 (this amounts to taking a generic element in the span of the tensors corresponding to the listed variables).

Here is the containment and singularity table:
\begin{center}
\begin{tabular}{c|c|c|c|c|c|c|c|c|c|c|c|}
& $\Obar_0$ & $\Obar_1$ & $\Obar_2$ & $\Obar_3$ & $\Obar_4$ & $\Obar_5$ & $\Obar_6$ & $\Obar_7$ & $\Obar_8$ & $\Obar_9$ & $\Obar_{10}$\\ \hline \hline
$\OO_0$ & ns & s & s & s & s & s & s & s & s & s & ns\\ \hline
$\OO_1$ & & ns & s & ns & s & s & s & s & s & s & ns\\ \hline
$\OO_2$ & & & ns & ns & ns & s & s & s & s & s & ns\\ \hline
$\OO_3$ & & & & ns & & & s & s & s & s & ns\\ \hline
$\OO_4$ & & & & & ns & & s & s & s & s & ns\\ \hline
$\OO_5$ & & & & & & ns & & & s & s & ns\\ \hline
$\OO_6$ & & & & & & & ns & ns & s & s & ns\\ \hline
$\OO_7$ & & & & & & & & ns & & s & ns\\ \hline
$\OO_8$ & & & & & & & & & ns & s & ns\\ \hline
$\OO_9$ & & & & & & & & & & ns & ns\\ \hline
$\OO_{10}$ & & & & & & & & & & & ns\\ \hline
\end{tabular}
\end{center}
We denote the free $A$-module $\Sc_{(a,b)} E^* \otimes \Sc_{(c,d,e)} F^* \otimes A(-a-b)$ by $(a,b;c,d,e)$.

\subsubsection{The orbit $\OO_9$}\label{F42O9}

The orbit closure $\Obar_9$ is not normal. The expected 
resolution for the coordinate ring of the normalization of
$\Obar_9$ is
\[
A \oplus (1, 1; 2, 1, 1) \oplus (2, 1; 2, 2, 2) \leftarrow (2, 2; 4, 2, 2) \leftarrow 0
\]
This orbit closure is a hypersurface, so there is only one differential $d_1$. The defining equation of $\Obar_9$ is the determinant of $d_1$. Alternatively this equation can be obtained as the discriminant of the determinant of a generic $3\times 3$ symmetric matrix of linear forms in two variables, which is a homogeneous polynomial of degree 12. Explicitly:
	\begin{align*}
	\delta  & = \det
	\begin{pmatrix}
	u x_{111} + v x_{211} & u x_{112} + v x_{212} & u x_{113} + v x_{213}\\
	u x_{112} + v x_{212} & u x_{122} + v x_{222} & u x_{123} + v x_{223}\\
	u x_{113} + v x_{213} & u x_{123} + v x_{223} & u x_{133} + v x_{233}
	\end{pmatrix} =\\
	& = a_{3,0} u^3 + a_{2,1} u^2 v + a_{1,2} u v^2 + a_{0,3} v^3,
	\end{align*}
and
	\[\disc (\delta) = 27 a_{3,0}^2 a_{1,2}^2 + 4 a_{3,0} a_{1,2}^3 + 4 a_{2,1}^3 a_{0,3} - a_{2,1}^2 a_{1,2}^2 -18 a_{3,0} a_{2,1} a_{1,2} a_{0,3}.\]
The block $(2, 2; 4, 2, 2) \rightarrow (2, 1; 2, 2, 2)$ of $d_1$ was constructed as the map
\begin{center}
  \begin{tikzpicture}
    \matrix (m) [matrix of math nodes, row sep=10pt, column sep=30pt,
    text height=1.5ex, text depth=0.25ex, ampersand replacement=\&]
    {\Sc_2 F^* \& E \otimes E^* \otimes \Sc_2 F^* \& E \otimes A_1\\};
    \path[->,semithick,font=\scriptsize]
    (m-1-1) edge node[auto] {$\tr^{(1)}$} (m-1-2);
    \path[right hook->,semithick,font=\scriptsize]
    (m-1-2) edge node[auto] {} (m-1-3);
  \end{tikzpicture}
\end{center}
The block $(2, 2; 4, 2, 2) \rightarrow (1, 1; 2, 1, 1)$ was constructed by taking the map
\begin{center}
  \begin{tikzpicture}
    \matrix (m) [matrix of math nodes, row sep=20pt, column sep=35pt,
	text height=1.5ex, text depth=0.25ex, ampersand replacement=\&]
    {\bigwedge^3 F^* \otimes \Sc_2 F^* \\
    F^* \otimes F^* \otimes F^* \otimes F^* \otimes F^* \\
    F^* \otimes \Sc_2 F^* \otimes \Sc_2 F^*\\
    F^* \otimes \bigwedge^2 \left( \Sc_2 F^* \right)\\};
    \path[->,semithick,font=\scriptsize]
    (m-1-1) edge node[auto] {$\Delta \otimes \Delta$} (m-2-1)
    (m-2-1) edge node[auto] {$m_{2,4} \otimes m_{3,5}$} (m-3-1)
    (m-3-1) edge node[auto] {$m_{2,3}$} (m-4-1);
  \end{tikzpicture}
\end{center}
on the $F^*$ factor and then applying the embedding $\bigwedge^2 E^* \otimes \bigwedge^2 \left( \Sc_2 F^* \right) \hookrightarrow A_2$ given by
\[
e_1^* \wedge e_2^* \otimes f_i^* f_j^* \wedge f_k^* f_l^* \longmapsto \det
	\begin{pmatrix}
	x_{1ij} & x_{1kl}\\
	x_{2ij} & x_{2kl}
	\end{pmatrix}.
\]
The last block $(2, 2; 4, 2, 2) \rightarrow A$ can be obtained as follows. First construct the map
\begin{center}
  \begin{tikzpicture}
    \matrix (m) [matrix of math nodes, row sep=20pt, column sep=35pt,
	text height=1.5ex, text depth=0.25ex, ampersand replacement=\&]
    {\bigwedge^3 F^* \otimes \bigwedge^3 F^* \otimes \Sc_2 F^*\\
    F^* \otimes F^* \otimes F^* \otimes F^* \otimes F^* \otimes F^* \otimes F^* \otimes F^*\\
    \Sc_2 F^* \otimes \Sc_2 F^* \otimes \Sc_2 F^* \otimes \Sc_2 F^*\\
    \bigwedge^2 \left( \Sc_2 F^* \right) \otimes \bigwedge^2 \left( \Sc_2 F^* \right)\\};
    \path[->,semithick,font=\scriptsize]
    (m-1-1) edge node[auto] {$\Delta \otimes \Delta \otimes \Delta$} (m-2-1)
    (m-2-1) edge node[auto] {$m_{1,7} \otimes m_{2,4} \otimes m_{3,5} \otimes m_{6,8}$} (m-3-1)
    (m-3-1) edge node[auto] {$m_{1,2} \otimes m_{3,4}$} (m-4-1);
  \end{tikzpicture}
\end{center}
on the $F^*$ factor. Then embed into $A$ via the map
\begin{center}
	\begin{tikzpicture}
	\matrix (m) [matrix of math nodes, row sep=20pt, column sep=35pt,
	text height=1.5ex, text depth=0.25ex, ampersand replacement=\&]
	{\bigwedge^2 E^* \otimes \bigwedge^2 \left( \Sc_2 F^* \right) \otimes \bigwedge^2 E^* \otimes \bigwedge^2 \left( \Sc_2 F^* \right) \\
	A_2 \otimes A_2\\
	A_4\\};
	\path[->,semithick,font=\scriptsize]
	(m-1-1) edge node[auto] {} (m-2-1)
	(m-2-1) edge node[auto] {} (m-3-1);
	\end{tikzpicture}
\end{center}
where the first step uses the embedding described earlier twice and the second step is symmetric multiplication.

The Betti table for the resolution of the normalization is
\[\begin{matrix}
&0&1\\\text{total:}&6&6\\\text{0:}&1&\text{.}\\\text{1:}&\text{.}&\text{.}\\\text{2:}&3&\text{.}\\\text{3:}&2&6\\\end{matrix}
\]
and the Betti table for the resolution of the cokernel $C(9)$ of the inclusion $\CC [\Obar_9] \hookrightarrow \CC [\mathcal{N} (\Obar_9)]$ is
\[\begin{matrix}
&0&1&2\\\text{total:}&5&6&1\\\text{2:}&3&\text{.}&\text{.}\\\text{3:}&2&6&\text{.}\\\text{4:}&\text{.}&\text{.}&\text{.}\\\text{5:}&\text{.}&\text{.}&\text{.}\\\text{6:}&\text{.}&\text{.}&\text{.}\\\text{7:}&\text{.}&\text{.}&\text{.}\\\text{8:}&\text{.}&\text{.}&\text{.}\\\text{9:}&\text{.}&\text{.}&\text{.}\\\text{10:}&\text{.}&\text{.}&1\\\end{matrix}
\]

\subsubsection{The orbit $\OO_8$}

The orbit closure $\Obar_8$ is not normal. The expected resolution for the coordinate ring of the normalization $\mathcal{N} (\Obar_8)$ is
\[
  A \oplus (1,1;2,1,1) \leftarrow (3,2;4,3,3) \oplus
  (2,2;4,2,2) \leftarrow (3,3;5,4,3) \leftarrow 0
\]
We construct explicitly the differential $d_2$ as follows. The first block $(3,3;5,4,3) \rightarrow (3,2;4,3,3)$ is defined on the $F^*$ factor as
\begin{center}
  \begin{tikzpicture}
    \matrix (m) [matrix of math nodes, row sep=10pt, column sep=50pt,
    text height=1.5ex, text depth=0.25ex, ampersand replacement=\&]
    {F^* \otimes F \& F^* \otimes F \otimes F^* \otimes F \& \Sc_2 F^* \otimes \bigwedge^2 F \\};
    \path[->,semithick,font=\scriptsize]
    (m-1-1) edge node[auto] {$\tr^{(1)}$} (m-1-2);
    \path[right hook->,semithick,font=\scriptsize]
    (m-1-2) edge node[auto] {$m_{1,3} \otimes m_{2,4}$} (m-1-3);
  \end{tikzpicture}
\end{center}
restricted to the subspace of traceless tensors $\Sc_{(2,1)} F^*$. On the $E^*$ factor, it is simply the trace map $\CC \rightarrow E^* \otimes E$, and putting the factors together we have
\begin{center}
  \begin{tikzpicture}
    \matrix (m) [matrix of math nodes, row sep=10pt, column sep=30pt,
    text height=1.5ex, text depth=0.25ex, ampersand replacement=\&]
    {\Sc_{(2,1)} F^* \& E^* \otimes E \otimes \Sc_2 F^* \otimes \bigwedge^2 F \& E \otimes \bigwedge^2 F \otimes A_1\\};
    \path[->,semithick,font=\scriptsize]
    (m-1-1) edge node[auto] {} (m-1-2);
    \path[right hook->,semithick,font=\scriptsize]
    (m-1-2) edge node[auto] {} (m-1-3);
  \end{tikzpicture}
\end{center}
The second block $(3,3;5,4,3) \rightarrow (2,2;4,2,2)$ is defined on the $F^*$ factor as
\begin{center}
  \begin{tikzpicture}
    \matrix (m) [matrix of math nodes, row sep=20pt, column sep=35pt,
	text height=1.5ex, text depth=0.25ex, ampersand replacement=\&]
    { F^* \otimes F \& \\
    F^* \otimes F \otimes F^* \otimes F \otimes F^* \otimes F \otimes F^* \otimes F \& \\
    \Sc_2 F^* \otimes \bigwedge^2 F \otimes \bigwedge^2 F \otimes \Sc_2 F^* \& \\
    \bigwedge^2 \left( \Sc_2 F^* \right) \otimes F^* \otimes F^* \&\\
    \bigwedge^2 \left( \Sc_2 F^* \right) \otimes \Sc_2 F^* \\ };
    \path[->,semithick,font=\scriptsize]
    (m-1-1) edge node[auto] {$\tr^{(1)} \otimes \tr^{(1)} \otimes \tr^{(1)}$} (m-2-1)
    (m-2-1) edge node[auto] {$m_{1,3} \otimes m_{2,8} \otimes m_{4,6} \otimes m_{5,7}$} (m-3-1)
    (m-3-1) edge node[auto] {$m_{1,4} \otimes {}^* \otimes {}^*$} (m-4-1)
    (m-4-1) edge node[auto] {$m_{2,3}$} (m-5-1);
  \end{tikzpicture}
\end{center}
restricted to the subspace of traceless tensors $\Sc_{(2,1)} F^*$. On the $E^*$ factor, we only have $\bigwedge^2 E^*$, so taking the factors together we get the map
\begin{center}
  \begin{tikzpicture}
    \matrix (m) [matrix of math nodes, row sep=10pt, column sep=30pt,
    text height=1.5ex, text depth=0.25ex, ampersand replacement=\&]
    {\Sc_{(2,1)} F^* \& \bigwedge^2 E^* \otimes \bigwedge^2 \left( \Sc_2 F^* \right) \otimes \Sc_2 F^* \& \Sc_2 F^* \otimes A_2\\};
    \path[->,semithick,font=\scriptsize]
    (m-1-1) edge node[auto] {} (m-1-2);
    \path[right hook->,semithick,font=\scriptsize]
    (m-1-2) edge node[auto] {} (m-1-3);
  \end{tikzpicture}
\end{center}
where the embedding into $A_2$ is the one described in \ref{F42O9}.
The Betti table for the normalization is
\[\begin{matrix}
&0&1&2\\\text{total:}&4&12&8\\\text{0:}&1&\text{.}&\text{.}\\\text{1:}&\text{.}&\text{.}&\text{.}\\\text{2:}&3&\text{.}&\text{.}\\\text{3:}&\text{.}&6&\text{.}\\\text{4:}&\text{.}&6&8\\\end{matrix}
\]
Dropping the row with entries of degree 4 and 5 in the differential $d_1$, we obtain a map
\[(3,2;4,3,3) \oplus (2,2;4,2,2) \rightarrow (1,1;2,1,1)\]
This is a presentation for the cokernel $C(8)$ of the inclusion $\CC [\Obar_8] \hookrightarrow \CC [\mathcal{N} (\Obar_8)]$; the Betti table for $C(8)$ is
\[\begin{matrix}
&0&1&2&3\\\text{total:}&3&12&11&2\\\text{0:}&3&\text{.}&\text{.}&\text{.}\\\text{1:}&\text{.}&6&\text{.}&\text{.}\\\text{2:}&\text{.}&6&11&\text{.}\\\text{3:}&\text{.}&\text{.}&\text{.}&\text{.}\\\text{4:}&\text{.}&\text{.}&\text{.}&2\\\end{matrix}
\]

\begin{remark}
In this particular example, the (truncated) cone procedure in M2 did not produce a result in a reasonable time. However it is easy to determine the representations appearing in the resolution of $C(8)$, namely:
\begin{gather*}
  (1,1;2,1,1) \leftarrow (3,2;4,3,3) \oplus
  (2,2;4,2,2) \leftarrow\\
  \leftarrow (3,3;5,4,3) \oplus (4,2;4,4,4) \leftarrow (5,4;6,6,6) \leftarrow 0
\end{gather*}
From this and the resolution of $\CC [\mathcal{N} (\Obar_8)]$, it is possible to construct and minimize the mapping cone by hand. The resulting complex is
\[(0,0;0,0,0)\leftarrow (4,2;4,4,4)\leftarrow (5,4;6,6,6) \leftarrow 0\]
The last map in the complex can be built as follows:
\begin{center}
  \begin{tikzpicture}
    \matrix (m) [matrix of math nodes, row sep=20pt, column sep=35pt,
	text height=1.5ex, text depth=0.25ex, ampersand replacement=\&]
    { E^* \otimes \bigwedge^2 E^* \otimes \bigwedge^2 E^* \otimes \bigwedge^3 F^* \otimes \bigwedge^3 F^* \& \\
    E^* \otimes E^* \otimes E^* \otimes E^* \otimes E^* \otimes F^* \otimes F^* \otimes F^* \otimes F^* \otimes F^* \otimes F^* \& \\
    \Sc_2 E^* \otimes E^* \otimes E^* \otimes E^* \otimes \Sc_2 F^* \otimes \Sc_2 F^* \otimes \Sc_2 F^* \& \\
    \Sc_2 E^* \otimes A_1 \otimes A_1 \otimes A_1 \&\\
    \Sc_2 E^* \otimes A_3 \\ };
    \path[->,semithick,font=\scriptsize]
    (m-1-1) edge node[auto] {$\Delta \otimes \Delta \otimes \Delta \otimes \Delta$} (m-2-1)
    (m-2-1) edge node[auto] {$m_{2,4} \otimes m_{6,9} \otimes m_{7,10} \otimes m_{8,11}$} (m-3-1)
    (m-3-1) edge node[auto] {$m_{2,5} \otimes m_{3,6} \otimes m_{4,7}$} (m-4-1)
    (m-4-1) edge node[auto] {$m_{2,3}$} (m-5-1);
  \end{tikzpicture}
\end{center}
and then resolved as usual to obtain the defining equations of $\Obar_8$.
\end{remark}

The Betti for the resolution of $\CC [\Obar_8]$ is
\[\begin{matrix}
&0&1&2\\\text{total:}&1&3&2\\\text{0:}&1&\text{.}&\text{.}\\\text{1:}&\text{.}&\text{.}&\text{.}\\\text{2:}&\text{.}&\text{.}&\text{.}\\\text{3:}&\text{.}&\text{.}&\text{.}\\\text{4:}&\text{.}&\text{.}&\text{.}\\\text{5:}&\text{.}&3&\text{.}\\\text{6:}&\text{.}&\text{.}&\text{.}\\\text{7:}&\text{.}&\text{.}&2\\\end{matrix}\]
It follows that $\Obar_8$ is Cohen-Macaulay.

\subsubsection{The orbit $\OO_7$}

The orbit closure $\Obar_7$ is not normal. The expected resolution for the coordinate ring of the normalization $\mathcal{N} (\Obar_7)$ is
\[
  A \oplus (1,1;2,2,0) \leftarrow (2,1;3,2,1) \leftarrow (3,1;3,3,2) \leftarrow 0
\]
We construct the differential $d_2 : (3,1;3,3,2) \rightarrow (2,1;3,2,1)$ explicitly as follows. On the $E^*$ factor, take the diagonalization $\Sc_2 E^* \rightarrow E^* \otimes E^*$. On the $F^*$ factor, take the map
\begin{center}
  \begin{tikzpicture}
    \matrix (m) [matrix of math nodes, row sep=20pt, column sep=35pt,
	text height=1.5ex, text depth=0.25ex, ampersand replacement=\&]
    { \bigwedge^2 F^* \\
    F \otimes F^* \otimes \bigwedge^2 F^* \\
    F \otimes F^* \otimes F^* \otimes F^* \\
    F \otimes F^* \otimes \Sc_2 F^* \\ };
    \path[->,semithick,font=\scriptsize]
    (m-1-1) edge node[auto] {$\tr^{(1)}$} (m-2-1)
    (m-2-1) edge node[auto] {$\Delta$} (m-3-1)
    (m-3-1) edge node[auto] {$m_{2,3}$} (m-4-1);
  \end{tikzpicture}
\end{center}
then project the first two factors onto the space of traceless tensors $\Sc_{(2,1)} F^*$ via the map $F \otimes F^* \rightarrow F \otimes F^* / \im (\tr^{(1)})$. Altogether we have
\begin{center}
  \begin{tikzpicture}
    \matrix (m) [matrix of math nodes, row sep=10pt, column sep=20pt,
    text height=1.5ex, text depth=0.25ex, ampersand replacement=\&]
    {\Sc_2 E^* \otimes \bigwedge^2 F^* \& E^* \otimes E^* \otimes \Sc_{(2,1)} F^* \otimes \Sc_2 F^* \& E^* \otimes \Sc_{(2,1)} F^* \otimes A_1\\};
    \path[->,semithick,font=\scriptsize]
    (m-1-1) edge node[auto] {} (m-1-2);
    \path[right hook->,semithick,font=\scriptsize]
    (m-1-2) edge node[auto] {} (m-1-3);
  \end{tikzpicture}
\end{center}
The Betti table for the normalization is
\[\begin{matrix}
&0&1&2\\\text{total:}&7&16&9\\\text{0:}&1&\text{.}&\text{.}\\\text{1:}&\text{.}&\text{.}&\text{.}\\\text{2:}&6&16&9\\\end{matrix}\]
Dropping the row of degree 3 in the differential $d_1$, we obtain a map
\[(2,1;3,2,1) \rightarrow (1,1;2,2,0)\]
This is a presentation for the cokernel $C(7)$ of the inclusion $\CC [\Obar_7] \hookrightarrow \CC [\mathcal{N} (\Obar_7)]$; the Betti table for $C(7)$ is
\[\begin{matrix}
&0&1&2&3&4&5\\\text{total:}&6&16&29&36&21&4\\\text{0:}&6&16&9&\text{.}&\text{.}&\text{.}\\\text{1:}&\text{.}&\text{.}&\text{.}&\text{.}&\text{.}&\text{.}\\\text{2:}&\text{.}&\text{.}&20&36&21&4\\\end{matrix}\]
By the cone procedure, we recover the resolution of $\CC [\Obar_7]$ which has the following Betti table
\[\begin{matrix}
&0&1&2&3&4\\\text{total:}&1&20&36&21&4\\\text{0:}&1&\text{.}&\text{.}&\text{.}&\text{.}\\\text{1:}&\text{.}&\text{.}&\text{.}&\text{.}&\text{.}\\\text{2:}&\text{.}&\text{.}&\text{.}&\text{.}&\text{.}\\\text{3:}&\text{.}&\text{.}&\text{.}&\text{.}&\text{.}\\\text{4:}&\text{.}&\text{.}&\text{.}&\text{.}&\text{.}\\\text{5:}&\text{.}&20&36&21&4\\\end{matrix}
\]
We observe that $\Obar_7$ is not Cohen-Macaulay because it has codimension 2 but its coordinate ring has projective dimension 4.

\subsubsection{The orbit $\OO_6$}

The orbit closure $\Obar_6$ is not normal. The expected resolution for the coordinate ring of the normalization $\mathcal{N} (\Obar_6)$ is
\begin{gather*}
  A \oplus (1,1;2,2,0) \leftarrow (2,1;3,2,1) \oplus (2,1;2,2,2) \leftarrow \\
  \leftarrow (3,3;5,4,3) \oplus (3,1;3,3,2) \leftarrow (4,3;5,5,4) \leftarrow 0
\end{gather*}
We construct the differential
\[d_1 : (2,1;3,2,1) \oplus (2,1;2,2,2) \rightarrow (1,1;2,2,0) \oplus A\]
explicitly. Notice how the domain of $d_1$ is isomorphic to $E^* \otimes F \otimes F^*$. It is clear that the representation on the $E^*$ factor is, up to a power of the determinant, simply $E^*$. On the $F^*$ factor, we have $\Sc_{(3,2,1)} F^* \oplus \Sc_{(2,2,2)} F^*$, with the first summand corresponding to the space of traceless tensors in $F\otimes F^*$ and the second factor corresponding to the tensor with non zero trace.

For the first block $E^* \otimes F \otimes F^* \rightarrow (1,1;2,2,0)$, take the map
\begin{center}
  \begin{tikzpicture}
    \matrix (m) [matrix of math nodes, row sep=20pt, column sep=35pt,
	text height=1.5ex, text depth=0.25ex, ampersand replacement=\&]
    { E^* \otimes F \otimes F^* \\
    E^* \otimes F \otimes F^* \otimes F \otimes F^* \\
    E^* \otimes \Sc_2 F \otimes \Sc_2 F^* \\
    \Sc_2 F \otimes A_1 \\ };
    \path[->,semithick,font=\scriptsize]
    (m-1-1) edge node[auto] {$\tr^{(1)}$} (m-2-1)
    (m-2-1) edge node[auto] {$m_{2,4} \otimes m_{3,5}$} (m-3-1)
    (m-3-1) edge node[auto] {} (m-4-1);
  \end{tikzpicture}
\end{center}

For the second block $E^* \otimes F \otimes F^* \rightarrow A_3$, take the map
\begin{center}
  \begin{tikzpicture}
    \matrix (m) [matrix of math nodes, row sep=20pt, column sep=35pt,
	text height=1.5ex, text depth=0.25ex, ampersand replacement=\&]
    { F \otimes F^* \otimes \bigwedge^3 F^*\\
    \bigwedge^2 F^* \otimes F^* \otimes \bigwedge^3 F^* \\
    F^* \otimes F^* \otimes F^* \otimes F^* \otimes F^* \otimes F^* \\
    \Sc_2 F^* \otimes \Sc_2 F^* \otimes \Sc_2 F^* \\
    \bigwedge^2 \left( \Sc_2 F^* \right) \otimes \Sc_2 F^* \\};
    \path[->,semithick,font=\scriptsize]
    (m-1-1) edge node[auto] {${}^*$} (m-2-1)
    (m-2-1) edge node[auto] {$\Delta \otimes \Delta$} (m-3-1)
    (m-3-1) edge node[auto] {$m_{1,4} \otimes m_{2,5} \otimes m_{3,6}$} (m-4-1)
    (m-4-1) edge node[auto] {$m_{1,2}$} (m-5-1);
  \end{tikzpicture}
\end{center}
on the $F^*$ factor. On the $E^*$ factor, simply tensor by a power of the determinant.
Altogether we have
\begin{center}
  \begin{tikzpicture}
    \matrix (m) [matrix of math nodes, row sep=10pt, column sep=30pt,
    text height=1.5ex, text depth=0.25ex, ampersand replacement=\&]
    {E^* \otimes \bigwedge^2 E^* \otimes \bigwedge^2 \left( \Sc_2 F^* \right) \otimes \Sc_2 F^* \& A_1 \otimes A_2 \& A_3\\};
    \path[right hook->,semithick,font=\scriptsize]
    (m-1-1) edge node[auto] {} (m-1-2);
    \path[->,semithick,font=\scriptsize]
    (m-1-2) edge node[auto] {} (m-1-3);
  \end{tikzpicture}
\end{center}
where the first step uses the embedding defined in \ref{F42O9} and the second step is symmetric multiplication.

The Betti table for the normalization is
\[\begin{matrix}
&0&1&2&3\\\text{total:}&7&18&17&6\\\text{0:}&1&\text{.}&\text{.}&\text{.}\\\text{1:}&\text{.}&\text{.}&\text{.}&\text{.}\\\text{2:}&6&18&9&\text{.}\\\text{3:}&\text{.}&\text{.}&\text{.}&\text{.}\\\text{4:}&\text{.}&\text{.}&8&6\\\end{matrix}\]
Dropping the row of degree 3 in the differential $d_1$, we obtain a map
\[(2,1;3,2,1) \oplus (2,1;2,2,2) \rightarrow (1,1;2,2,0)\]
This is a presentation for the cokernel $C(6)$ of the inclusion $\CC [\Obar_6] \hookrightarrow \CC [\mathcal{N} (\Obar_6)]$; the Betti table for $C(6)$ is
\[\begin{matrix}
&0&1&2&3&4\\\text{total:}&6&18&15&6&3\\\text{2:}&6&18&15&\text{.}&\text{.}\\\text{3:}&\text{.}&\text{.}&\text{.}&\text{.}&\text{.}\\\text{4:}&\text{.}&\text{.}&\text{.}&6&3\\\end{matrix}\]
By the cone procedure, we recover the resolution of $\CC [\Obar_6]$ which has the following Betti table
\[\begin{matrix}
&0&1&2&3\\\text{total:}&1&6&8&3\\\text{0:}&1&\text{.}&\text{.}&\text{.}\\\text{1:}&\text{.}&\text{.}&\text{.}&\text{.}\\\text{2:}&\text{.}&\text{.}&\text{.}&\text{.}\\\text{3:}&\text{.}&6&\text{.}&\text{.}\\\text{4:}&\text{.}&\text{.}&8&\text{.}\\\text{5:}&\text{.}&\text{.}&\text{.}&3\\\end{matrix}
\]
We observe that $\Obar_6$ is Cohen-Macaulay.

\subsubsection{The orbit $\OO_5$}

The orbit closure $\Obar_5$ is not normal. The expected resolution for the coordinate ring of the normalization $\mathcal{N} (\Obar_5)$ is
\begin{gather*}
  A \oplus (1,0;1,1,0) \leftarrow (3,0;2,2,2) \oplus (1,1;3,1,0) \oplus (2,0;2,1,1) \leftarrow \\
  \leftarrow (3,1;4,2,2) \oplus (2,2;3,3,2) \oplus (2,1;4,1,1) \leftarrow \\
  \leftarrow (3,2;5,3,2) \leftarrow (3,3;5,5,2) \leftarrow 0
\end{gather*}
We construct the differential $d_4 : (3,3;5,5,2) \rightarrow (3,2;5,3,2)$ explicitly. Start with the map
\begin{center}
  \begin{tikzpicture}
    \matrix (m) [matrix of math nodes, row sep=20pt, column sep=35pt,
	text height=1.5ex, text depth=0.25ex, ampersand replacement=\&]
    { \Sc_3 F \\
    E \otimes E^* \otimes \Sc_3 F \otimes F \otimes F^* \otimes F \otimes F^* \\
    E \otimes E^* \otimes \Sc_3 F \otimes \Sc_2 F \otimes \Sc_2 F^*\\
    E \otimes \Sc_3 F \otimes \Sc_2 F \otimes A_1\\ };
    \path[->,semithick,font=\scriptsize]
    (m-1-1) edge node[auto] {$\tr^{(1)} \otimes \tr^{(1)} \otimes \tr^{(1)}$} (m-2-1)
    (m-2-1) edge node[auto] {$m_{4,6} \otimes m_{5,7}$} (m-3-1)
    (m-3-1) edge node[auto] {} (m-4-1);
  \end{tikzpicture}
\end{center}
Then project onto $E \otimes \Sc_{(3,2)} F \otimes A_1$ modding out $\Sc_3 F \otimes \Sc_2 F$ by the image of the map
\begin{center}
  \begin{tikzpicture}
    \matrix (m) [matrix of math nodes, row sep=10pt, column sep=30pt,
    text height=1.5ex, text depth=0.25ex, ampersand replacement=\&]
    {\Sc_4 F \otimes F \& \Sc_3 F \otimes F \otimes F \& \Sc_3 F \otimes \Sc_2 F\\};
    \path[->,semithick,font=\scriptsize]
    (m-1-1) edge node[auto] {$\Delta$} (m-1-2)
    (m-1-2) edge node[auto] {$m_{2,3}$} (m-1-3);
  \end{tikzpicture}
\end{center}

The Betti table for the normalization is
\[\begin{matrix}
&0&1&2&3&4\\\text{total:}&7&28&41&30&10\\\text{0:}&1&\text{.}&\text{.}&\text{.}&\text{.}\\\text{1:}&6&24&20&\text{.}&\text{.}\\\text{2:}&\text{.}&4&21&30&10\\\end{matrix}\]
Dropping the row of degree 2 in the differential $d_1$, we obtain a map
\[(3,0;2,2,2) \oplus (1,1;3,1,0) \oplus (2,0;2,1,1) \rightarrow (1,0;1,1,0)\]
This is a presentation for the cokernel $C(5)$ of the inclusion $\CC [\Obar_5] \hookrightarrow \CC [\mathcal{N} (\Obar_5)]$; the Betti table for $C(5)$ is
\[\begin{matrix}
&0&1&2&3&4&5&6&7&8\\\text{total:}&6&24&56&163&298&285&144&33&1\\\text{1:}&6&24&20&\text{.}&\text{.}&\text{.}&\text{.}&\text{.}&\text{.}\\\text{2:}&\text{.}&\text{.}&36&42&10&\text{.}&\text{.}&\text{.}&\text{.}\\\text{3:}&\text{.}&\text{.}&\text{.}&121&288&285&144&33&\text{.}\\\text{4:}&\text{.}&\text{.}&\text{.}&\text{.}&\text{.}&\text{.}&\text{.}&\text{.}&1\\\end{matrix}\]
By the cone procedure, we recover the resolution of $\CC [\Obar_5]$ which has the following Betti table
\[\begin{matrix}
&0&1&2&3&4&5&6&7\\\text{total:}&1&19&133&288&285&144&33&1\\\text{0:}&1&\text{.}&\text{.}&\text{.}&\text{.}&\text{.}&\text{.}&\text{.}\\\text{1:}&\text{.}&\text{.}&\text{.}&\text{.}&\text{.}&\text{.}&\text{.}&\text{.}\\\text{2:}&\text{.}&4&\text{.}&\text{.}&\text{.}&\text{.}&\text{.}&\text{.}\\\text{3:}&\text{.}&15&12&\text{.}&\text{.}&\text{.}&\text{.}&\text{.}\\\text{4:}&\text{.}&\text{.}&121&288&285&144&33&\text{.}\\\text{5:}&\text{.}&\text{.}&\text{.}&\text{.}&\text{.}&\text{.}&\text{.}&1\\\end{matrix}\]
We observe that $\Obar_5$ is not Cohen-Macaulay because it has codimension 4 but its coordinate ring has projective dimension 7.

\subsubsection{The orbit $\OO_4$}
The orbit closure $\Obar_4$ is normal with rational singularities. The expected resolution for the coordinate ring $\CC [\Obar_4]$ is
\begin{gather*}
	A \leftarrow (3,0;2,2,2) \oplus (2,1;3,2,1) \leftarrow \\
	\leftarrow (2,2;3,3,2) \oplus (3,1;3,3,2) \oplus (3,1;4,2,2) \oplus (2,2;4,3,1) \leftarrow \\
	\leftarrow (3,2;5,3,2) \oplus (3,2;4,3,3) \leftarrow (3,3;6,3,3) \leftarrow 0
\end{gather*}
The differential $d_4$ was written explicitly. For the block, $(3,3;6,3,3) \rightarrow (3,2;4,3,3)$ take the map
\begin{center}
  \begin{tikzpicture}
    \matrix (m) [matrix of math nodes, row sep=10pt, column sep=45pt,
    text height=1.5ex, text depth=0.25ex, ampersand replacement=\&]
    {\Sc_3 F^* \& E \otimes E^* \otimes \Sc_2 F^* \otimes F^* \& E \otimes F^* \otimes A_1\\};
    \path[->,semithick,font=\scriptsize]
    (m-1-1) edge node[auto] {$\tr^{(1)} \otimes \Delta$} (m-1-2);
    \path[right hook->,semithick,font=\scriptsize]
    (m-1-2) edge node[auto] {} (m-1-3);
  \end{tikzpicture}
\end{center}
For the block $(3,3;6,3,3) \rightarrow (3,2;5,3,2)$ construct first the map
\begin{center}
  \begin{tikzpicture}
    \matrix (m) [matrix of math nodes, row sep=20pt, column sep=35pt,
	text height=1.5ex, text depth=0.25ex, ampersand replacement=\&]
    { \Sc_3 F^* \otimes  \bigwedge^3 F^* \\
    E \otimes E^* \otimes \Sc_2 F^* \otimes F^* \otimes F^* \otimes F^* \otimes F^* \\
    E \otimes E^* \otimes \Sc_3 F^* \otimes \Sc_2 F^* \otimes F^*\\
    E \otimes \Sc_3 F^* \otimes F^* \otimes A_1\\ };
    \path[->,semithick,font=\scriptsize]
    (m-1-1) edge node[auto] {$\tr^{(1)} \otimes \Delta \otimes \Delta$} (m-2-1)
    (m-2-1) edge node[auto] {$m_{3,5} \otimes m_{4,6}$} (m-3-1)
    (m-3-1) edge node[auto] {} (m-4-1);
  \end{tikzpicture}
\end{center}
Finally project onto $E \otimes \Sc_{(3,1)} F^* \otimes A_1$ by modding out $\Sc_3 F^* \otimes F^*$ by the image of the map
\begin{center}
  \begin{tikzpicture}
    \matrix (m) [matrix of math nodes, row sep=10pt, column sep=30pt,
    text height=1.5ex, text depth=0.25ex, ampersand replacement=\&]
    {\Sc_4 F^* \& \Sc_3 F^* \otimes F^*\\};
    \path[->,semithick,font=\scriptsize]
    (m-1-1) edge node[auto] {$\Delta$} (m-1-2);
  \end{tikzpicture}
\end{center}
The Betti table for the resolution is
\[\begin{matrix}
&0&1&2&3&4\\\text{total:}&1&20&45&36&10\\\text{0:}&1&\text{.}&\text{.}&\text{.}&\text{.}\\\text{1:}&\text{.}&\text{.}&\text{.}&\text{.}&\text{.}\\\text{2:}&\text{.}&20&45&36&10\\\end{matrix}\]
We conclude that $\Obar_4$ is \CM{}. The defining equations are the $3\times 3$ minors of the generic matrix of a linear map $E\otimes F \rightarrow F^*$ after symmetrizing indices on the $F^*$ side.

\subsubsection{The orbit $\OO_3$}\label{F42O3}
The orbit closure $\Obar_3$ is normal with rational singularities. It is degenerate with equations given by the $2\times 2$ minors of
\[
	\begin{pmatrix}
	x_{111} & x_{211} \\
	x_{112} & x_{212} \\
	x_{113} & x_{213} \\
	x_{122} & x_{222} \\
	x_{123} & x_{223} \\
	x_{133} & x_{233} 
	\end{pmatrix}
\]
the generic matrix of a linear map $E\rightarrow \Sc_2 F^*$. The Betti table for the resolution is
\[\begin{matrix}
&0&1&2&3&4&5\\\text{total:}&1&15&40&45&24&5\\\text{0:}&1&\text{.}&\text{.}&\text{.}&\text{.}&\text{.}\\\text{1:}&\text{.}&15&40&45&24&5\\\end{matrix}\]
It follows that $\Obar_3$ is Cohen-Macaulay.

\subsubsection{The orbit $\OO_2$}\label{F42O2}
The orbit closure $\Obar_2$ is normal with rational singularities. It is degenerate with equations given by the $2\times 2$ minors of the generic matrix of a linear map $E\rightarrow \Sc_2 F^*$ together with the coefficients of the determinant of a generic $3\times 3$ matrix of linear forms in two variables. The former are the equations of $\Obar_3$ while the latter are the coefficients $a_{3,0}, a_{2,1}, a_{1,2}, a_{0,3}$ of $\delta$ as defined in \ref{F42O9}.
The Betti table for the resolution is
\[\begin{matrix}
&0&1&2&3&4&5&6\\\text{total:}&1&19&58&75&44&11&2\\\text{0:}&1&\text{.}&\text{.}&\text{.}&\text{.}&\text{.}&\text{.}\\\text{1:}&\text{.}&15&40&45&24&5&\text{.}\\\text{2:}&\text{.}&4&18&30&20&\text{.}&\text{.}\\\text{3:}&\text{.}&\text{.}&\text{.}&\text{.}&\text{.}&6&2\\\end{matrix}\]
It follows that $\Obar_2$ is Cohen-Macaulay.

\subsubsection{The orbit $\OO_1$}\label{F42O1}
The orbit closure $\Obar_1$ is normal with rational singularities. It is degenerate with equations given by the $2\times 2$ minors of the generic matrix of a linear map $E\rightarrow \Sc_2 F^*$ together with the coefficients of the $2\times 2$ minors of a generic $3\times 3$ matrix of linear forms in two variables. The former are the equations of $\Obar_3$ while the latter are the coefficients of the $2\times 2$ minors of the matrix defined in \ref{F42O9}.
The Betti table for the resolution is
\[\begin{matrix}
&0&1&2&3&4&5&6&7&8\\\text{total:}&1&33&144&294&336&210&69&16&3\\\text{0:}&1&\text{.}&\text{.}&\text{.}&\text{.}&\text{.}&\text{.}&\text{.}&\text{.}\\\text{1:}&\text{.}&33&144&294&336&210&48&\text{.}&\text{.}\\\text{2:}&\text{.}&\text{.}&\text{.}&\text{.}&\text{.}&\text{.}&21&16&3\\\end{matrix}\]
It follows that $\Obar_1$ is Cohen-Macaulay.

\subsection{The case $(F_{4}, \alpha_3)$}
\label{F43}

The representation is $E \otimes F$, where $E=\CC^2$ and
$F=\CC^3$; the group acting is $\SL (E) \times \SL (F) \times
\CC^\times$. The orbit closures for this representation are classical determinantal varieties. For a description of the minimal free resolutions of their coordinate rings, the reader can consult \cite[Ch. 6]{MR1988690}, for example.

\subsection{The case $(F_{4}, \alpha_4)$}
\label{F44}

The representation is $V=V(\omega_3,B_3)$, the third fundamental representation of the group $\SO (7,\CC)$; the group acting is $\Spin (7)$. The corresponding polynomial ring is
\[A=\CC [x_{1231},x_{1221},x_{1121},x_{0121},x_{1111},x_{0111},x_{0011},x_{0001}] = \Sym (V^*).\]
The variables in $A$ are indexed by the roots in $\gg_1$. The variables are weight vectors in $V^*$, with the following weights:
\[
\begingroup
\renewcommand*{\arraystretch}{1.3}
\begin{array}{lcl}
x_{1231} \leftrightarrow \frac{1}{2} (\epsilon_1 + \epsilon_2 + \epsilon_3) & & x_{1221} \leftrightarrow \frac{1}{2} (\epsilon_1 + \epsilon_2 - \epsilon_3)\\
x_{1121} \leftrightarrow \frac{1}{2} (\epsilon_1 - \epsilon_2 + \epsilon_3) & & x_{0121} \leftrightarrow \frac{1}{2} (-\epsilon_1 + \epsilon_2 + \epsilon_3)\\
x_{1111} \leftrightarrow \frac{1}{2} (\epsilon_1 - \epsilon_2 - \epsilon_3) & & x_{0111} \leftrightarrow \frac{1}{2} (- \epsilon_1 + \epsilon_2 - \epsilon_3)\\
x_{0011} \leftrightarrow \frac{1}{2} (- \epsilon_1 - \epsilon_2 + \epsilon_3) & & x_{0001} \leftrightarrow \frac{1}{2} (- \epsilon_1 - \epsilon_2 - \epsilon_3)
\end{array}
\endgroup
\]

In characteristic zero, the representation has the following orbits, listed along with the dimension of the closure and a representative:
\begin{center}
\begin{tabular}{ccc}
orbit & dimension & representative\\ \hline
$\OO_0$ & 0& 0\\
$\OO_1$ & 7& $x_{0001} = 1$\\
$\OO_2$ & 8& $x_{0001} = x_{1231} =1$
\end{tabular}
\end{center}
All the orbit closures are normal, Cohen-Macaulay, Gorenstein and have rational singularities. Here is the containment and singularity table:
\begin{center}
\begin{tabular}{c|c|c|c|}
& $\Obar_0$ & $\Obar_1$ & $\Obar_2$\\ \hline \hline
$\OO_0$ & ns & s & ns\\ \hline
$\OO_1$ & & ns & ns\\ \hline
$\OO_2$ & & & ns\\ \hline
\end{tabular}
\end{center}

\subsubsection{The orbit $\OO_1$}

The variety $\Obar_1$ is the closure of the highest weight vector orbit. It is a hypersurface defined by an invariant of degree 2. The invariant was described explicitly by Igusa \cite{MR0277558}.
\section{Representations of type $G_2$}\label{type_G2}
In this section, we will analyze the cases corresponding to gradings on the simple Lie algebra of type $G_2$. Each case corresponds to the choice of a distinguished node on the Dynkin diagram for $G_2$. The nodes are numbered according to the conventions in Bourbaki \cite{MR0240238}.
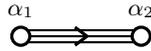
\begin{center}
\begin{tikzpicture}[scale=0.8]
	\triplelinearrowright{(0,0)}
	\rootnode{(0,0)}{$\alpha_1$}
	\rootnode{(2,0)}{$\alpha_2$}
\end{tikzpicture}
\end{center}

\subsection{The case $(G_{2}, \alpha_1)$}
\label{G21}

The representation is $E=\CC^2$ and the group acting is $\GL (E)$. The only orbits in this case are the origin and the dense orbit.

\subsection{The case $(G_{2}, \alpha_2)$}
\label{G22}

The representation is $\Sc_3 E$, where $E=\CC^2$; the group acting is $\GL (E)$. The corresponding polynomial ring is
\[A=\QQ [x_{ijk} \mid 1\leq i\leq j\leq 2] = \Sym (\Sc_3 E^*).\]

In characteristic zero, the representation has the following orbits, listed along with the dimension of the closure and a representative:
\begin{center}
\begin{tabular}{ccc}
orbit & dimension & representative\\ \hline
$\OO_0$ & 0& 0\\
$\OO_1$ & 2& $x_{111} = 1$\\
$\OO_2$ & 3& $x_{112} = 1$\\
$\OO_3$ & 4& $x_{111} = x_{222} = 1$
\end{tabular}
\end{center}
Here is the containment and singularity table:
\begin{center}
\begin{tabular}{c|c|c|c|c|}
& $\Obar_0$ & $\Obar_1$ & $\Obar_2$ & $\Obar_3$\\ \hline \hline
$\OO_0$ & ns & s & s & ns\\ \hline
$\OO_1$ & & ns & s & ns\\ \hline
$\OO_2$ & & & ns & ns\\ \hline
$\OO_3$ & & & & ns\\ \hline
\end{tabular}
\end{center}
We denote the free $A$-module $\Sc_{(a,b)} E^* \otimes A(-(a+b)/3)$ by $(a,b)$.

\subsubsection{The orbit $\OO_2$}\label{G22O2}
The orbit closure $\Obar_2$ is not normal. The expected equivariant resolution for the coordinate ring of the normalization of $\Obar_2$ is
\[A \oplus (2,1) \leftarrow (4,2) \leftarrow 0\]
The orbit closure is a hypersurface, so there is only one differential $d_1$. The defining equation of $\Obar_9$ is the determinant of $d_1$. Alternatively this equation can be obtained using the cone procedure.

The block $(4,2)\rightarrow (2,1)$ of $d_1$ was constructed by taking the map
\begin{center}
  \begin{tikzpicture}
    \matrix (m) [matrix of math nodes, row sep=10pt, column sep=35pt,
    text height=1.5ex, text depth=0.25ex, ampersand replacement=\&]
    {\bigwedge^2 E^* \otimes \Sc_2 E^* \& E^* \otimes E^* \otimes E^* \otimes E^* \& E^* \otimes A_1\\};
    \path[->,semithick,font=\scriptsize]
    (m-1-1) edge node[auto] {$\Delta\otimes\Delta$} (m-1-2);
    \path[->,semithick,font=\scriptsize]
    (m-1-2) edge node[auto] {$m_{2,3,4}$} (m-1-3);
  \end{tikzpicture}
\end{center}

The block $(4, 2) \rightarrow A$ can be obtained as follows:
\begin{center}
  \begin{tikzpicture}
    \matrix (m) [matrix of math nodes, row sep=20pt, column sep=35pt,
	text height=1.5ex, text depth=0.25ex, ampersand replacement=\&]
    {\bigwedge^2 E^* \otimes \bigwedge^2 E^* \otimes \Sc_2 E^* \\
    E^* \otimes E^* \otimes E^* \otimes E^* \otimes E^* \otimes E^*\\
    A_1 \otimes A_1\\
    A_2\\};
    \path[->,semithick,font=\scriptsize]
    (m-1-1) edge node[auto] {$\Delta \otimes \Delta \otimes \Delta$} (m-2-1)
    (m-2-1) edge node[auto] {$m_{1,3,5} \otimes m_{2,4,6}$} (m-3-1)
    (m-3-1) edge node[auto] {} (m-4-1);
  \end{tikzpicture}
\end{center}
where the last step is symmetric multiplication.

The Betti table for the resolution of the normalization is
\[\begin{matrix}
&0&1\\\text{total:}&3&3\\\text{0:}&1&\text{.}\\\text{1:}&2&3\\\end{matrix}
\]
and the Betti table for the resolution of the cokernel $C(2)$ of the inclusion $\CC[\Obar_2] \hookrightarrow \CC[\mathcal{N} (\Obar_2)]$ is
\[
\begin{matrix}
&0&1&2\\\text{total:}&2&3&1\\\text{1:}&2&3&\text{.}\\\text{2:}&\text{.}&\text{.}&1\\
\end{matrix}
\]

\subsubsection{The orbit $\OO_1$}
The orbit closure $\Obar_1$ is normal with rational singularities. The expected equivariant resolution for the coordinate ring of $\CC[\Obar_1]$ is
\[A\leftarrow (4,2) \leftarrow (5,4) \leftarrow 0\]
Notice that the first differential is the second block of the differential described in \ref{G22O2}. Here we describe how to explicitly construct $d_2 : (5,4) \rightarrow (4,2)$, the second differential in the resolution. This is obtained by writing the map
\begin{center}
  \begin{tikzpicture}
    \matrix (m) [matrix of math nodes, row sep=20pt, column sep=35pt,
	text height=1.5ex, text depth=0.25ex, ampersand replacement=\&]
    {\bigwedge^2 E^* \otimes \bigwedge^2 E^* \otimes E^* \\
    E^* \otimes E^* \otimes E^* \otimes E^* \otimes E^* \\
    \Sc_2 E^* \otimes A_1\\};
    \path[->,semithick,font=\scriptsize]
    (m-1-1) edge node[auto] {$\Delta \otimes \Delta$} (m-2-1)
    (m-2-1) edge node[auto] {$m_{2,3} \otimes m_{1,4,5}$} (m-3-1);
  \end{tikzpicture}
\end{center}
The Betti table for the resolution is
\[
\begin{matrix}
&0&1&2\\
\text{total:}&1&3&2\\
\text{0:}&1&\text{.}&\text{.}\\
\text{1:}&\text{.}&3&2\\
\end{matrix}
\]
It follows that the orbit closure $\Obar_1$ is \CM{}.
\appendix
\section{Equivariant maps}\label{eq_maps}
In this appendix, we give a brief description of the equivariant maps between representations that are used to construct the differentials in our complexes (more details can be found in \cite[Ch. 1]{MR1988690}). The notation introduced here is used extensively throughout sections \ref{type_E6}, \ref{type_F4} and \ref{type_G2}. We use the symbol $\Sc_\lambda$ to denote the Schur functor associated to the partition $\lambda$. In particular, $\Sc_i = \Sc_{(i)} = \Sym_i$ denotes the $i$-th symmetric power. In order to simplify the notation as much as possible, we omit to write any symbol for the identity and other obvious maps.

Throughout this appendix, $E$ denotes a complex vector space of finite dimension $n$ with basis $\{e_1,\ldots,e_n\}$; $E^*$ denotes the dual of $E$ and we take $\{e^*_1,\ldots,e^*_n\}$ to be the basis dual to $\{e_1,\ldots,e_n\}$. All the maps we describe are $\GL (E)$-equivariant.

\subsection{Diagonals}
Let $r,s$ be natural numbers such that $0 \leq r+s \leq n$. The \emph{exterior diagonal} is the map:
\[\Delta : \bigwedge^{r+s} E \longrightarrow \bigwedge^r E \otimes \bigwedge^s E\]
with
\[\Delta(e_1 \wedge \ldots \wedge e_{r+s}) = \sum_{\sigma\in\mathfrak{S}^{r,s}_{r+s}} (-1)^{\sgn (\sigma)} e_{\sigma (1)} \wedge \ldots \wedge e_{\sigma (r)} \otimes e_{\sigma (r+1)} \wedge \ldots \wedge e_{\sigma (r+s)}.\]
Here
\[\mathfrak{S}^{r,s}_{r+s} := \{ \sigma \in \mathfrak{S}_{r+s} \mid \sigma (1) < \ldots < \sigma (r), \sigma (r+1) < \ldots < \sigma (r+s)\}\]
where $\mathfrak{S}_{d}$ denotes the symmetric group on $d$ letters and $\sgn (\sigma)$ is the sign of the permutation $\sigma$.

Now let $r,s$ be arbitrary natural numbers. The \emph{symmetric diagonal} is the map:
\[\Delta : \Sc_{r+s} E \longrightarrow \Sc_r E \otimes \Sc_s E\]
with
\[\Delta(e_1 \ldots e_{r+s}) = \sum_{\sigma\in\mathfrak{S}^{r,s}_{r+s}} e_{\sigma (1)} \ldots e_{\sigma (r)} \otimes e_{\sigma (r+1)} \ldots e_{\sigma (r+s)}.\]

Both diagonals can be generalized to the case where the codomain is a tensor product of more than two exterior or symmetric powers of $E$ in the obvious way. It will be clear from the context whether we are using the exterior or symmetric diagonal and how many and what factors we are taking in the codomain.

\subsection{Multiplications}
Let $r,s$ be natural numbers such that $0 \leq r+s \leq n$. The \emph{exterior multiplication} is the map:
\[m : \bigwedge^r E \otimes \bigwedge^s E \longrightarrow \bigwedge^{r+s} E \]
with
\[m (u_1 \wedge \ldots \wedge u_r \otimes v_1 \wedge \ldots \wedge v_s) = u_1 \wedge \ldots \wedge u_r \wedge v_1 \wedge \ldots \wedge v_s .\]

Now let $r,s$ be arbitrary natural numbers. The \emph{symmetric multiplication} is the map:
\[m : \Sc_r E \otimes \Sc_s E \longrightarrow \Sc_{r+s} E\]
with
\[m (u_1 \ldots u_r \otimes v_1 \ldots v_s) = u_1 \ldots u_r v_1 \ldots v_s.\]

Both multiplications can be generalized to the case where the domain is a tensor product of more than two exterior or symmetric powers of $E$ in the obvious way. It will be clear from the context whether we are using the exterior or symmetric multiplication and how many and what factors we are taking in the domain.

When the tensor factors that we wish to multiply are not adjacent, we use subscripts to clarify which factors we are multiplying. For example, we write $m_{1,3} : E\otimes E\otimes E \rightarrow \Sc_2 E \otimes E$ to indicate we apply the symmetric multiplication to the first and third factor, leaving the second one alone.

\subsection{Traces}
Let $r$ be a natural number such that $0 \leq r \leq n$. The \emph{exterior trace} is the map:
\[\tr^{(r)} : \CC \longrightarrow \bigwedge^r E \otimes \bigwedge^r E^*\]
with
\[\tr^{(r)}(1) = \sum_{1\leq i_1 < \ldots < i_r \leq n} e_{i_1} \wedge \ldots \wedge e_{i_r} \otimes e^*_{i_1} \wedge \ldots \wedge e^*_{i_r}.\]

Now let $r$ be an arbitrary natural number. The \emph{symmetric trace} is the map:
\[\tr^{(r)} : \CC \longrightarrow \Sc_r E \otimes \Sc_r E^*\]
with
\[\tr^{(r)}(1) = \sum_{1\leq i_1 \leq \ldots \leq i_r \leq n} e_{i_1}  \ldots  e_{i_r} \otimes e^*_{i_1}  \ldots  e^*_{i_r}.\]

It will be clear from the context whether we are using the exterior or symmetric trace.

\subsection{Exterior duality}
Let $r$ be a natural number such that $0 \leq r \leq n$. The \emph{exterior duality} is the map:
\[{}^* : \bigwedge^r E \longrightarrow \bigwedge^{n-r} E^*\]
with
\[{}^* (e_{i_1} \wedge \ldots \wedge e_{i_r}) = (-1)^{\sgn (\sigma)} e^*_{j_1} \wedge \ldots \wedge e^*_{j_{n-r}},\]
where $1\leq i_1 < \ldots < i_r \leq n$, $1\leq j_1 < \ldots < j_{n-r} \leq n$, $\{i_1,\ldots, i_r\} \cup \{j_1,\ldots, j_{n-r}\} = \{1,\ldots, n\}$ and $\sigma$ is the permutation
\[
\begin{pmatrix}
1 & \ldots & r & r+1 & \ldots & n\\
i_1 & \ldots & i_r & j_1 & \ldots & j_{n-r}
\end{pmatrix}.
\]
Similarly we can define an exterior duality ${}^* : \bigwedge^r E^* \longrightarrow \bigwedge^{n-r} E$.

\bibliographystyle{plain}
\bibliography{/Users/fez/BibTeX/math.bib}
\addcontentsline{toc}{section}{References}

\end{document}